\documentclass[11pt]{amsart}
\usepackage{amsmath}
\usepackage{amsfonts}
\usepackage{eucal}
\usepackage{oldgerm}
\usepackage{latexsym}
\usepackage{amsmath}
\usepackage{amscd}
\usepackage{amssymb} 
\usepackage{amsmath}
\usepackage{amsthm}
\usepackage{latexsym}
\usepackage{arydshln}

\usepackage{oldgerm}
\usepackage{latexsym}
\usepackage{amsmath}
\usepackage{amscd}
\usepackage{amssymb} 
\usepackage{amsmath}
\usepackage{amsthm}
\usepackage{latexsym}
\usepackage{arydshln}

\makeatletter
\@namedef{subjclassname@2020}{%
  \textup{2020} Mathematics Subject Classification}
\makeatother

\begin{document}

\title [Twisted Koecher-Maass series of the Ikeda type lift]{Twisted Koecher-Maass series of the Ikeda type lift\\ for the exceptional group of type $E_{7,3}$}  


\author{Hidenori Katsurada and Henry H. Kim}

\date{July 2 2022}

\thanks{The first author is partially supported by JSPS KAKENHI Grant Number No. 21K03152. The second author is partially supported by NSERC grant \#482564.}
\address{Hidenori Katsurada\\
Department of Mathematics,  Hokkaido University, Sapporo, Japan \\
and Muroran Insititute of Technology \\
27-1 Mizumoto, Muroran 050-8585, Japan
}
\email{hidenori@mmm.muroran-it.ac.jp}
\address{Henry H. Kim \\
Department of mathematics \\
 University of Toronto \\
Toronto, Ontario M5S 2E4, CANADA \\
and Korea Institute for Advanced Study, Seoul, KOREA}
\email{henrykim@math.toronto.edu}

\maketitle


\newcommand{\alp}{\alpha}
\newcommand{\bet}{\beta}
\newcommand{\gam}{\gamma}
\newcommand{\del}{\delta}
\newcommand{\eps}{\epsilon}
\newcommand{\zet}{\zeta}
\newcommand{\tht}{\theta}
\newcommand{\iot}{\iota}
\newcommand{\kap}{\kappa}
\newcommand{\lam}{\lambda}
\newcommand{\sig}{\sigma}
\newcommand{\ups}{\upsilon}
\newcommand{\ome}{\omega}
\newcommand{\vep}{\varepsilon}
\newcommand{\vth}{\vartheta}
\newcommand{\vpi}{\varpi}
\newcommand{\vrh}{\varrho}
\newcommand{\vsi}{\varsigma}
\newcommand{\vph}{\varphi}
\newcommand{\Gam}{\Gamma}
\newcommand{\Del}{\Delta}
\newcommand{\Tht}{\Theta}
\newcommand{\Lam}{\Lambda}
\newcommand{\Sig}{\Sigma}
\newcommand{\Ups}{\Upsilon}
\newcommand{\Ome}{\Omega}


\newcommand{\frka}{{\mathfrak a}}    \newcommand{\frkA}{{\mathfrak A}}
\newcommand{\frkb}{{\mathfrak b}}    \newcommand{\frkB}{{\mathfrak B}}
\newcommand{\frkc}{{\mathfrak c}}    \newcommand{\frkC}{{\mathfrak C}}
\newcommand{\frkd}{{\mathfrak d}}    \newcommand{\frkD}{{\mathfrak D}}
\newcommand{\frke}{{\mathfrak e}}    \newcommand{\frkE}{{\mathfrak E}}
\newcommand{\frkf}{{\mathfrak f}}    \newcommand{\frkF}{{\mathfrak F}}
\newcommand{\frkg}{{\mathfrak g}}    \newcommand{\frkG}{{\mathfrak G}}
\newcommand{\frkh}{{\mathfrak h}}    \newcommand{\frkH}{{\mathfrak H}}
\newcommand{\frki}{{\mathfrak i}}    \newcommand{\frkI}{{\mathfrak I}}
\newcommand{\frkj}{{\mathfrak j}}    \newcommand{\frkJ}{{\mathfrak J}}
\newcommand{\frkk}{{\mathfrak k}}    \newcommand{\frkK}{{\mathfrak K}}
\newcommand{\frkl}{{\mathfrak l}}    \newcommand{\frkL}{{\mathfrak L}}
\newcommand{\frkm}{{\mathfrak m}}    \newcommand{\frkM}{{\mathfrak M}}
\newcommand{\frkn}{{\mathfrak n}}    \newcommand{\frkN}{{\mathfrak N}}
\newcommand{\frko}{{\mathfrak o}}    \newcommand{\frkO}{{\mathfrak O}}
\newcommand{\frkp}{{\mathfrak p}}    \newcommand{\frkP}{{\mathfrak P}}
\newcommand{\frkq}{{\mathfrak q}}    \newcommand{\frkQ}{{\mathfrak Q}}
\newcommand{\frkr}{{\mathfrak r}}    \newcommand{\frkR}{{\mathfrak R}}
\newcommand{\frks}{{\mathfrak s}}    \newcommand{\frkS}{{\mathfrak S}}
\newcommand{\frkt}{{\mathfrak t}}    \newcommand{\frkT}{{\mathfrak T}}
\newcommand{\frku}{{\mathfrak u}}    \newcommand{\frkU}{{\mathfrak U}}
\newcommand{\frkv}{{\mathfrak v}}    \newcommand{\frkV}{{\mathfrak V}}
\newcommand{\frkw}{{\mathfrak w}}    \newcommand{\frkW}{{\mathfrak W}}
\newcommand{\frkx}{{\mathfrak x}}    \newcommand{\frkX}{{\mathfrak X}}
\newcommand{\frky}{{\mathfrak y}}    \newcommand{\frkY}{{\mathfrak Y}}
\newcommand{\frkz}{{\mathfrak z}}    \newcommand{\frkZ}{{\mathfrak Z}}


\newcommand{\bfa}{{\mathbf a}}    \newcommand{\bfA}{{\mathbf A}}
\newcommand{\bfb}{{\mathbf b}}    \newcommand{\bfB}{{\mathbf B}}
\newcommand{\bfc}{{\mathbf c}}    \newcommand{\bfC}{{\mathbf C}}
\newcommand{\bfd}{{\mathbf d}}    \newcommand{\bfD}{{\mathbf D}}
\newcommand{\bfe}{{\mathbf e}}    \newcommand{\bfE}{{\mathbf E}}
\newcommand{\bff}{{\mathbf f}}    \newcommand{\bfF}{{\mathbf F}}
\newcommand{\bfg}{{\mathbf g}}    \newcommand{\bfG}{{\mathbf G}}
\newcommand{\bfh}{{\mathbf h}}    \newcommand{\bfH}{{\mathbf H}}
\newcommand{\bfi}{{\mathbf i}}    \newcommand{\bfI}{{\mathbf I}}
\newcommand{\bfj}{{\mathbf j}}    \newcommand{\bfJ}{{\mathbf J}}
\newcommand{\bfk}{{\mathbf k}}    \newcommand{\bfK}{{\mathbf K}}
\newcommand{\bfl}{{\mathbf l}}    \newcommand{\bfL}{{\mathbf L}}
\newcommand{\bfm}{{\mathbf m}}    \newcommand{\bfM}{{\mathbf M}}
\newcommand{\bfn}{{\mathbf n}}    \newcommand{\bfN}{{\mathbf N}}
\newcommand{\bfo}{{\mathbf o}}    \newcommand{\bfO}{{\mathbf O}}
\newcommand{\bfp}{{\mathbf p}}    \newcommand{\bfP}{{\mathbf P}}
\newcommand{\bfq}{{\mathbf q}}    \newcommand{\bfQ}{{\mathbf Q}}
\newcommand{\bfr}{{\mathbf r}}    \newcommand{\bfR}{{\mathbf R}}
\newcommand{\bfs}{{\mathbf s}}    \newcommand{\bfS}{{\mathbf S}}
\newcommand{\bft}{{\mathbf t}}    \newcommand{\bfT}{{\mathbf T}}
\newcommand{\bfu}{{\mathbf u}}    \newcommand{\bfU}{{\mathbf U}}
\newcommand{\bfv}{{\mathbf v}}    \newcommand{\bfV}{{\mathbf V}}
\newcommand{\bfw}{{\mathbf w}}    \newcommand{\bfW}{{\mathbf W}}
\newcommand{\bfx}{{\mathbf x}}    \newcommand{\bfX}{{\mathbf X}}
\newcommand{\bfy}{{\mathbf y}}    \newcommand{\bfY}{{\mathbf Y}}
\newcommand{\bfz}{{\mathbf z}}    \newcommand{\bfZ}{{\mathbf Z}}


\newcommand{\cala}{{\mathcal A}}
\newcommand{\calb}{{\mathcal B}}
\newcommand{\calc}{{\mathcal C}}
\newcommand{\cald}{{\mathcal D}}
\newcommand{\cale}{{\mathcal E}}
\newcommand{\calf}{{\mathcal F}}
\newcommand{\calg}{{\mathcal G}}
\newcommand{\calh}{{\mathcal H}}
\newcommand{\cali}{{\mathcal I}}
\newcommand{\calj}{{\mathcal J}}
\newcommand{\calk}{{\mathcal K}}
\newcommand{\call}{{\mathcal L}}
\newcommand{\calm}{{\mathcal M}}
\newcommand{\caln}{{\mathcal N}}
\newcommand{\calo}{{\mathcal O}}
\newcommand{\calp}{{\mathcal P}}
\newcommand{\calq}{{\mathcal Q}}
\newcommand{\calr}{{\mathcal R}}
\newcommand{\cals}{{\mathcal S}}
\newcommand{\calt}{{\mathcal T}}
\newcommand{\calu}{{\mathcal U}}
\newcommand{\calv}{{\mathcal V}}
\newcommand{\calw}{{\mathcal W}}
\newcommand{\calx}{{\mathcal X}}
\newcommand{\caly}{{\mathcal Y}}
\newcommand{\calz}{{\mathcal Z}}


\newcommand{\scra}{{\mathscr A}}
\newcommand{\scrb}{{\mathscr B}}
\newcommand{\scrc}{{\mathscr C}}
\newcommand{\scrd}{{\mathscr D}}
\newcommand{\scre}{{\mathscr E}}
\newcommand{\scrf}{{\mathscr F}}
\newcommand{\scrg}{{\mathscr G}}
\newcommand{\scrh}{{\mathscr H}}
\newcommand{\scri}{{\mathscr I}}
\newcommand{\scrj}{{\mathscr J}}
\newcommand{\scrk}{{\mathscr K}}
\newcommand{\scrl}{{\mathscr L}}
\newcommand{\scrm}{{\mathscr M}}
\newcommand{\scrn}{{\mathscr N}}
\newcommand{\scro}{{\mathscr O}}
\newcommand{\scrp}{{\mathscr P}}
\newcommand{\scrq}{{\mathscr Q}}
\newcommand{\scrr}{{\mathscr R}}
\newcommand{\scrs}{{\mathscr S}}
\newcommand{\scrt}{{\mathscr T}}
\newcommand{\scru}{{\mathscr U}}
\newcommand{\scrv}{{\mathscr V}}
\newcommand{\scrw}{{\mathscr W}}
\newcommand{\scrx}{{\mathscr X}}
\newcommand{\scry}{{\mathscr Y}}
\newcommand{\scrz}{{\mathscr Z}}


\newcommand{\AAA}{{\mathbb A}} 
\newcommand{\BB}{{\mathbb B}}
\newcommand{\CC}{{\mathbb C}}
\newcommand{\DD}{{\mathbb D}}
\newcommand{\EE}{{\mathbb E}}
\newcommand{\FF}{{\mathbb F}}
\newcommand{\GG}{{\mathbb G}}
\newcommand{\HH}{{\mathbb H}}
\newcommand{\II}{{\mathbb I}}
\newcommand{\JJ}{{\mathbb J}}
\newcommand{\KK}{{\mathbb K}}
\newcommand{\LL}{{\mathbb L}}
\newcommand{\MM}{{\mathbb M}}
\newcommand{\NN}{{\mathbb N}}
\newcommand{\OO}{{\mathbb O}}
\newcommand{\PP}{{\mathbb P}}
\newcommand{\QQ}{{\mathbb Q}}
\newcommand{\RR}{{\mathbb R}}
\newcommand{\SSS}{{\mathbb S}} 
\newcommand{\TT}{{\mathbb T}}
\newcommand{\UU}{{\mathbb U}}
\newcommand{\VV}{{\mathbb V}}
\newcommand{\WW}{{\mathbb W}}
\newcommand{\XX}{{\mathbb X}}
\newcommand{\YY}{{\mathbb Y}}
\newcommand{\ZZ}{{\mathbb Z}}


\newcommand{\tta}{\hbox{\tt a}}    \newcommand{\ttA}{\hbox{\tt A}}
\newcommand{\ttb}{\hbox{\tt b}}    \newcommand{\ttB}{\hbox{\tt B}}
\newcommand{\ttc}{\hbox{\tt c}}    \newcommand{\ttC}{\hbox{\tt C}}
\newcommand{\ttd}{\hbox{\tt d}}    \newcommand{\ttD}{\hbox{\tt D}}
\newcommand{\tte}{\hbox{\tt e}}    \newcommand{\ttE}{\hbox{\tt E}}
\newcommand{\ttf}{\hbox{\tt f}}    \newcommand{\ttF}{\hbox{\tt F}}
\newcommand{\ttg}{\hbox{\tt g}}    \newcommand{\ttG}{\hbox{\tt G}}
\newcommand{\tth}{\hbox{\tt h}}    \newcommand{\ttH}{\hbox{\tt H}}
\newcommand{\tti}{\hbox{\tt i}}    \newcommand{\ttI}{\hbox{\tt I}}
\newcommand{\ttj}{\hbox{\tt j}}    \newcommand{\ttJ}{\hbox{\tt J}}
\newcommand{\ttk}{\hbox{\tt k}}    \newcommand{\ttK}{\hbox{\tt K}}
\newcommand{\ttl}{\hbox{\tt l}}    \newcommand{\ttL}{\hbox{\tt L}}
\newcommand{\ttm}{\hbox{\tt m}}    \newcommand{\ttM}{\hbox{\tt M}}
\newcommand{\ttn}{\hbox{\tt n}}    \newcommand{\ttN}{\hbox{\tt N}}
\newcommand{\tto}{\hbox{\tt o}}    \newcommand{\ttO}{\hbox{\tt O}}
\newcommand{\ttp}{\hbox{\tt p}}    \newcommand{\ttP}{\hbox{\tt P}}
\newcommand{\ttq}{\hbox{\tt q}}    \newcommand{\ttQ}{\hbox{\tt Q}}
\newcommand{\ttr}{\hbox{\tt r}}    \newcommand{\ttR}{\hbox{\tt R}}
\newcommand{\tts}{\hbox{\tt s}}    \newcommand{\ttS}{\hbox{\tt S}}
\newcommand{\ttt}{\hbox{\tt t}}    \newcommand{\ttT}{\hbox{\tt T}}
\newcommand{\ttu}{\hbox{\tt u}}    \newcommand{\ttU}{\hbox{\tt U}}
\newcommand{\ttv}{\hbox{\tt v}}    \newcommand{\ttV}{\hbox{\tt V}}
\newcommand{\ttw}{\hbox{\tt w}}    \newcommand{\ttW}{\hbox{\tt W}}
\newcommand{\ttx}{\hbox{\tt x}}    \newcommand{\ttX}{\hbox{\tt X}}
\newcommand{\tty}{\hbox{\tt y}}    \newcommand{\ttY}{\hbox{\tt Y}}
\newcommand{\ttz}{\hbox{\tt z}}    \newcommand{\ttZ}{\hbox{\tt Z}}

\makeatletter 
\newcommand{\LEQQ}{\mathrel{\mathpalette\gl@align<}}
\newcommand{\GEQQ}{\mathrel{\mathpalette\gl@align>}}
\newcommand{\gl@align}[2]{\lower.6ex\vbox{\baselineskip\z@skip\lineskip\z@ 
\ialign{$\m@th#1\hfil##\hfil$\crcr#2\crcr=\crcr}}}
\makeatother

\newcommand{\ua}{\underline{a}}
\newcommand{\GK}{\mathrm{GK}}
\newcommand{\ord}{\mathrm{ord}}
\def\mattwo(#1;#2;#3;#4){\left(\begin{matrix}
                               #1 & #2 \\
                               #3  & #4
                                      \end{matrix}\right)}
\theoremstyle{plain}
\newtheorem{theorem}{Theorem}[section]
\newtheorem{lemma}[theorem]{Lemma}
\newtheorem{proposition}[theorem]{Proposition}
\newtheorem{definition}[theorem]{Definition}
\newtheorem{conjecture}{Conjecture}
\newtheorem{comment}{Comment}[section]
\theoremstyle{remark}
\newtheorem{remark}[theorem]{\bf {Remark}}
\newtheorem{corollary}[theorem]{\bf {Corollary}}
\newtheorem{formula}{\bf {Formula}}[section]
\newtheorem{question}{\bf {Question}}[section]
\numberwithin{equation}{section}


\begin{abstract}
We compute the twisted Koecher-Maass series of the first and second kind of the Ikeda type lift for the exceptional group of type $E_{7,3}$.
As an application, we obtain their rationality result.
\end{abstract}

\section{Introduction}

Let $\frkJ_\QQ$ be the exceptional Jordan algebra consisting of 
$3 \times 3$ matrices with entries in the Cayley numbers, and $\calm'$ the group scheme over $\ZZ$ of type $E_{6,2}$. (See Section \ref{excep} for the definitions.)  Let $\frkT$ be the exceptional domain in $\CC^{27}$. Let ${\bf G}$ be a connected reductive group of type $GE_{7,3}$.  For a cusp form 
$$F(Z)=\sum_{T \in \frkJ(\ZZ)_{>0}}a_F(T)\exp(2\pi \sqrt{-1}(T,Z))$$ of weight $k$ ($k$ even) with respect to ${\bf G}(\ZZ)$ and a Dirichlet character $\chi$, 
we define the twisted Koecher-Maass series $K^{(2)}(s,F,\chi)$ of the second kind as 
$$K^{(2)}(s,F,\chi)=\sum_{T\in {\frak J}(\Bbb Z)_{>0}/\calm'(\ZZ)} \frac {\chi(\det T)a_F(T)}{\epsilon(T) \det(T)^s},
$$
where for $T \in \frkJ(\ZZ)_{>0}$, $\epsilon(T)=\#(\calu_T(\ZZ))$, and $\calu_T$ denotes the group scheme of type $F_4$. (See Section \ref{F_4}.) 
For a positive integer $N$ let
$\phi_N:\calm'(\ZZ) \longrightarrow \calm'(\ZZ/N\ZZ)$ be the homomorphism induced by the natural surjection $\pi_N:\ZZ \longrightarrow \ZZ/N\ZZ$, and put $\calm'(N;\ZZ)=\mathrm{Ker}(\phi_N)$.
We define the twisted Koecher-Maass series $K^{(1)}(s,F,\chi)$ of the first kind as
$$K^{(1)}(s,F,\chi)=\sum_{T \in {\frak J}(\Bbb Z)_{>0}/\calm'(N;\ZZ)} \frac {\chi(\mathrm{Tr}(T)) a_F(T)}{\epsilon_N(T) \det(T)^s},
$$
where $\epsilon_N(T)=\#(\calm'(N;\ZZ) \cap \calu_T(\ZZ))$. 

For a cuspidal Hecke eigenform $f \in S_{2k-8}(SL_2(\ZZ))$ ($2k\geq 20$), let $F_f$ be the Ikeda type lift which is a cuspidal Hecke eigenform of weight $2k$ with respect to ${\bf G}(\ZZ)$ constructed by Kim and Yamauchi (\cite{K-Y}). In this paper we give explicit formulas of $K^{(1)}(s,F_f,\chi)$ and $K^{(2)}(s,F_f,\chi)$.

\begin{theorem} \label{th.explicit-KM-2} 
Let $\displaystyle c={5!7!11! \over (2\pi)^{28}}$ and $L(s,\pi_f,\chi)$ be the  $L$-function of the cuspidal automorphic representation $\pi_f$  attached to $f$ twisted by $\chi$. Then
\begin{align*}
K^{(2)}(s,F_f,\chi) =c \zeta(2)\zeta(6)\zeta(8)\zeta(12)\times \prod_{i=1}^3 L(s-k-9/2+4i-3,\pi_f,\chi).
\end{align*}
\end{theorem}
\begin{theorem} \label{th.explicit-KM-1}
Let $\chi$ be a primitive Dirichlet character mod $N$. Suppose that $\chi$ is not a quadratic character. Let $l=\mathrm{GCD}(3,\phi(N))$ with $\phi$ is the Euler phi-function, and let $u_0$ be a primitive $l$-th root of unity mod $N$.
\begin{itemize}
\item[(1)] Suppose that $\chi(u_0) \not=1$.  Then $K^{(1)}(s,F_f,\chi)=0$.
\item[(2)] Suppose that $\chi(u_0) =1$ and fix a character $\widetilde \chi$ such that $\chi=\widetilde \chi^3$ (cf. Lemma \ref{lem.existence-of-third-root-of-character}).  Then 
\begin{align*}
K^{(1)}(s,F_f,\chi) =c \zeta(2)\zeta(6)\zeta(8)\zeta(12)d_{N} \sum_{\eta \in \cald_{N}}J(\overline{\widetilde \chi \eta},\overline{\widetilde \chi \eta},\overline{\widetilde \chi \eta}) \prod_{i=1}^3  L(s-k-9/2+4i-3,\pi_f,\overline{\widetilde \chi \eta}),
\end{align*}
where $d_N=N^{64}\prod_{p|N} (1-p^{-2})(1-p^{-6})(1-p^{-8})(1-p^{-12})$, 
$$\cald_{N}=\{ \eta \ | \ \eta \text{ is a Dirichlet character } mod \ N \text{ such that } \eta^l=1 \},
$$
and $J(\xi,\xi,\xi)$ is the generalized Jacobi sum (cf. Section 4).
In particular, if $N$ is odd, we have
\begin{align*}
K^{(1)}(s,F_f,\chi) =c \zeta(2)\zeta(6)\zeta(8)\zeta(12)d_{N}\sum_{\eta \in \cald_{N}}{W(\overline{\widetilde \chi \eta})^3 \over W(\bar \chi)}\prod_{i=1}^3  L(s-k-9/2+4i-3,\pi_f,\overline{\widetilde \chi \eta}),
\end{align*}
where $W(\xi)$ is the Gauss sum of the Dirichlet character $\xi$.
\end{itemize}
\end{theorem}
As an application, we obtain the rationality result for $K^{(1)}(m,F_f,\chi)$ and $K^{(2)}(m,F_f,\chi)$ for $9\leq m\leq 2k-9$ (Theorems 7.2 and 7.3).


In the case of Siegel modular forms, Choie-Kohnen \cite{CK} studied the twisted Koecher-Maass series of the first kind and proved the analytic continuation and the functional equation, and studied their special values.
On the other hand, the first named author \cite{Kat2} studied the twisted Koecher-Maass series of both the first kind and the second kind, and gave explicit formulas for them associated with the Duke-Imamoglu-Ikeda lift.

Our method of proving the above two theorems is similar to that used in \cite{Kat2}. However, unlike the Siegel case, we cannot use explicit matrix decompositions in the exceptional group. It is one of the obstacles we need to overcome.
This paper is organized as follows. In Section 2, we briefly review modular forms on the exceptional domain. Unlike in previous papers \cite{Bai, K-Y}, we need to consider the exceptional similitude group $GE_{7,3}$. We use the definition of $GE_{7,3}$ in \cite{KY1}. 

In Section 3, we prove that the twisted Koecher-Maass series of the first kind for any cusp forms has analytic continuation and the functional equation. In Section 4, we obtain the relationship between the twisted Koecher-Maass series of the first kind and the second kind.

In Section 5, we briefly review the mass formula and local density $\beta_p(T)$ from \cite{K-K-Y}, where we used them to compute the Rankin-Selberg series $R(s,F_f,F_f)$ for the Ikeda type lift $F_f$. In Section 6, we review the Ikeda type lift. By construction, the Fourier coefficients of $F_f$ are expressed in terms of a product of the local Siegel series. Therefore, using the mass formula in Section 5, we can express the twisted Koecher-Maass series of the second kind
as an Euler product:
\[K^{(2)}(s,F_f,\chi)=c\prod_p H_p(\alpha_p,\chi(p)p^{-s+k+9/2}),\]
where $c$ is in Theorem \ref{th.explicit-KM-2}, $\alpha_p$ is the $p$-th Satake parameter for $f$, and $H_p(X,t)$ is a certain power series involving the Siegel series and the local density. It is remarkable that $K^{(2)}(s,F_f,\chi)$ has an Euler product since $K^{(2)}(s,F,\chi)$ does not have an Euler product for a general cusp form $F$. In Section 7, we prove Theorems \ref{th.explicit-KM-2} and \ref{th.explicit-KM-1}, and obtain rationality result of the twisted Koecher-Maass series.

\smallskip

\noindent {\bf Notation.} In addition to the standard symbols $\ZZ,\QQ,\RR$ and $\CC$, for a prime number $p$, let $\QQ_p$ and $\ZZ_p$ be the field of $p$-adic numbers and the ring of $p$-adic integers. For a commutative ring $R$ let $R^\times$ denote the group of units in $R$.

Let $\sim$ be an equivalence relation on a set $\cals$. We denote by $\cals/\sim$ the set of equivalence classes of $\cals$ under $\sim$. We use the same symbol $\cals/\sim$ to denote a complete set of representatives.
Let $G$ be a group acting on a set $\cals$. For two elements $a_1$ and $a_2$, we write $a_2 \sim_G a_1$ if $a_2=g\cdot a_1$ with $g \in G$. The relation $\sim_G$ is an equivalence relation on $\cals$ and we write $\cals/G$ instead of $\cals/\sim_G$.

For an associate or non-associate algebra $R$ let $M_{mn}(R)$ denote the set of $m \times n$ matrix with entries in $R$.
In particular we put $M_n(R)=M_{nn}(R)$. In particular if $R$ is a commutative ring, for an element $A \in M_n(R)$ let $\det A$ denote the determinant of $A$. We put $GL_n(R)=\{ A \in M_n(R) \ | \ \det A \in R^\times \}.$
Moreover, for an $R$-module, $M$, let $GL(M)$ denote the group of $R$-linear transformations on $M$. For square matrices $A_1,\ldots,A_r$, we write $A_1 \bot \cdots\bot A_r=\begin{pmatrix}A_1 & O       &  O \\
                                    O   & \ddots & O        \\
                                    O   &    O       &   A_r
\end{pmatrix}$. We sometimes write $\mathrm{diag}(A_1,\cdots,A_r)$ instead of $A_1 \bot \cdots \bot A_r$. For $x\in \Bbb R$, let $\bold e(x)=e^{2\pi \sqrt{-1} x}$. For an element $a$ of $\ZZ/N\ZZ$, we use the same symbol $a$ to denote its representative mod $N$. Moreover, for a Dirichlet character $\chi$ mod $N$, we use the same symbol $\chi$ to denote the mapping $\ZZ/N\ZZ \ni a \text{ mod } N \mapsto \chi(a) \in \CC$.

\section{Modular forms on the exceptional domain}\label{excep}

We will freely use the notations from \cite{Bai, K-Y, K-K-Y}. 
Let $\frkC_{\Bbb Q}$ and $\mathfrak o\subset \frkC_{\Bbb Q}$ be the Cayley numbers and integral Cayley numbers, resp. The trace and the norm on $\frkC_{\QQ}$ are defined by
\[\mathrm{Tr}(x)=x+\bar x \text{ and } N(x)=x\bar x,\]
where $\bar x$ is the anti-involution in \cite[Section 2] {K-K-Y}.
Let $\frkJ_\QQ$ be the exceptional Jordan algebra consisting of matrices 
\[X=(x_{ij})_{1 \le i,j \le 3}=\begin{pmatrix} a & x & y \\ \bar x & b & z \\ \bar y & \bar z & c \end{pmatrix}\]
with $a,b,c \in \QQ$ and $x,y,z \in \frkC_{\QQ}$. We define the determinant $\det X$ and the trace $\mathrm{Tr}(X)$ by
\[\det X=abc-aN(z)-bN(y)-cN(x)+\mathrm{Tr}((xz)\bar y), \quad  \mathrm{Tr}(x)=a+b+c.\]
We define a lattice $\frkJ(\ZZ)$ of $\frkJ_{\QQ}$  by 
\[\frkJ(\ZZ)=\{ X =(x_{ij}) \in \frkJ_{\QQ} \ | \ x_{ii} \in \ZZ \text{ and } x_{ij} \in \frko \text{ for } i\not=j\}. \]
For a commutative algebra $R$, we put $\frkJ(R)=\frkJ(\ZZ) \otimes_{\ZZ} R$. 
Recall
$$\frkJ(R)^{\rm ns}=\{ X \in \frkJ(R) \ | \ \det (X) \not=0\},\quad R^+_3(R)=\{ X^2 \ | \ X \in \frkJ(R)^{\rm ns}\}.
$$
We denote by $\overline{R^+_3(\RR)}$ the closure of $R^+_3(\RR)$ in $\frkJ(\RR) \simeq \RR^{27}$. For a subring $A$ of $\RR$, set
\[\frkJ(A)_{>0}=\frkJ(A) \cap R^+_3(\RR) \text{ and } \frkJ(A)_{\ge 0}=\frkJ(A) \cap \overline{R^+_3(\RR)}.\]
We also define $\frkJ_{2,\QQ}$ as the set of matrices of forms $X=\begin{pmatrix}  a & x \\ \bar x & b \end{pmatrix},$ $a,b \in \QQ,\, x \in \frkC_\QQ,$
and its lattice $\frkJ_2(\ZZ)$ as 
\[\frkJ_2(\ZZ)=\left\{\begin{pmatrix}  a & x \\ \bar x & b \end{pmatrix} \ | \ a,b \in \ZZ, x \in \mathfrak o\right\}.\]
We define the determinant $\det X$ and the trace $\mathrm{Tr}(X)$ of $X =\begin{pmatrix}  a & x \\ \bar x & b \end{pmatrix} \in \frkJ_{2,\QQ}$ by
\[\det X=ab-N(x), \quad  \mathrm{Tr}(x)=a+b.\]
For a commutative algebra $R$, put $\frkJ_2(R)=\frkJ_2(\ZZ) \otimes_\ZZ R$.

Recall the exceptional domain:
$$\frak T:=\{Z=X+Y\sqrt{-1}\in \frak J_\Bbb C\ |\ X,Y\in \frak J_\Bbb R,\ Y\in R^+_3(\Bbb R)\}
$$
which is a complex analytic subspace of $\Bbb C^{27}$.

Define the group schemes $\calm$ and $\calm'$ over $\ZZ$ by 
\[\calm(R)=\{ g \in GL(\frkJ(R)) \ | \ \det (g\cdot X)=\nu(g)\det X \text{ with } \nu(g) \in R^\times\},
\]
\[\calm'(R)=\{ g \in \calm(R) \ |\,  \nu(g) =1 \}.
\]
Put ${\bf M}=\calm \otimes_{\ZZ} \QQ$ and ${\bf M}'=\calm' \otimes_{\ZZ} \QQ$. Then ${\bf M}$ is an algebraic group over $\QQ$ of type $GE_{6,2}$ and ${\bf M}'$ is the derived group of ${\bf M}$, which is a simple group of type $E_{6,2}$.

Let $\bold X, \bold X'$ be two $\Bbb Q$-vector spaces, each isomorphic to $\frak J$, and $\Xi, \Xi'$ be copies of $\Bbb Q$. Let $\bold W=\bold X\oplus \Xi\oplus \bold X'\oplus \Xi'$, and for $w=(X,\xi,X',\xi')\in\bold W$,
define a quartic form $Q$ on $\bold W$ by 
$$Q(w)=(X\times X, X'\times X')-\xi \det(X)-\xi' \det(X')-\frac 14( (X,X')-\xi\xi')^2,
$$
and a skew-symmetric bilinear form $\{\,,\,\}$ by
$$\{w_1,w_2\}=(X_1,X_2')-(X_2,X_1')+\xi_1\xi_2'-\xi_2\xi_1'.
$$
Recall the definition of the exceptional group of type $GE_{7,3}$ \cite{KY1}:
$$\bold G(\Bbb Q)=\{ g\in GL(\bold W)\ |\, Qg(w)= \mu(g)^2 Q(w),\, g\{w_1,w_2\}=\mu(g) \{w_1,w_2\} \, \},
$$
for some $\mu(g)\in \Bbb Q^\times$. Then $\mu$ is a rational character of $\bold G$. Let $\bold G'(\Bbb Q)=\{ g\in \bold G(\QQ) | \, \mu(g)=1\}$. 
Define the similitude factor $h_0(a)$ as
\begin{equation}\label{similitude}
h_0(a)(X,\xi,X',\xi')=(aX,a^{-1}\xi,X',a^2\xi').
\end{equation}
Then $\mu(h_0(a))=a$, and $\bold G=\{h_0(a)\}\ltimes \bold G'$. Here $\bold G'$ is a simply connected algebraic $\Bbb Q$-group of type $E_{7,3}^{\rm sc}$ as in \cite{Bai}, and $\bold G$ is a $\Bbb Q$-group of type $GE_{7,3}$. 
Let $I_\bold W$ be the identity operator on $\bold W$. Let $\bold Z$ be the central torus of $GL(\bold W)$, i.e.,
$$\bold Z_\Bbb Q=\{ \lambda I_{\bold W} |\, \text{$\lambda\in \Bbb Q$, $\lambda\ne 0$}\}.
$$
Then $\bold Z$ is the central torus of $\bold G$, and $\bold G=\bold Z\cdot \bold G'=(\bold Z\times\bold G')/\mu_2$, where $\mu_2$ is embedded in both centers, and $\bold G/\bold Z$ is the adjoint exceptional group of type $E_{7,3}^{\rm ad}$. 
The real rank of $\bold G'$ is 3, and it is split over $\Bbb Q_p$ for any prime $p$.

Let $\widetilde{\bold M}=\bold Z\cdot \bold M$ (almost direct product). Then
$\bold P=\widetilde{\bold M}\bold N$ is the Siegel parabolic subgroup of $\bold G$. We extend the character $\nu$ to $\bold P$ by defining $\nu(zmn)=\nu(m)$ for $z\in\bold Z$, $m\in \bold M$, and $n\in\bold N$.

Let $\bold G(\Bbb R)^+=\{ g\in \bold G(\Bbb R) | \, \mu(g)>0\}$. 
For a function $F: \frak T\longrightarrow \Bbb C$, $k\in \Bbb Z_{\ge 0}$, and $g\in \bold G(\Bbb R)^+$, 
let $g=z g'$ with $z\in \bold Z(\Bbb R)$ and $z>0$, and $g'\in \bold G'(\Bbb R)$. 
We define the ``slash operator" by 
$$F|_k g(Z):=j(g',Z)^{-k}F(g'Z),
$$
where $j(g',Z)$ is the canonical factor of automorphy in \cite{Kim1}. Hence the center acts trivially.
We write $F|g$ instead of $F|_k g$ when there is no confusion.

Let $\Gamma=\bold G(\Bbb Z)$. 
For a positive integer $N$, define 
$\calm(N;\ZZ)$ to be the kernel of the map $\calm(\Bbb Z)\longrightarrow \calm(\Bbb Z/N\Bbb Z)$.

Let $F\in S_k(\Gamma)$ be a cusp form of weight $k$ ($k$ even) with respect to $\Gamma$, and write 
$$F(Z)=\sum_{T\in \frak J(\ZZ)_{>0}} a_F(T)\bold e( (T,Z)).
$$ 
For a primitive Dirichlet character $\chi$ mod $N$, let
$$F_\chi(Z)=\sum_{T\in \frak J(\ZZ)_{>0}} \chi({\rm tr}(T)) a_F(T) \bold e((T,Z)).
$$ 
Then for $m\in \calm(N;\ZZ)$, $F_\chi(mZ)=F_\chi(Z)$.

Let $W(\bar\chi)=\sum_{a \, \text{(mod $N$)}} \bar\chi(a) \bold e(a/N)$ be the Gauss sum associated to $\bar\chi$. 
Note that if $(a,N)>1$, $\bar\chi(a)=0$. Hence the sum is in fact over all $a$ (mod $N$) such that $(a,N)=1$.

\begin{theorem} \label{th.fc-twisted-F}
Let $h_0(N)\in \bold G(\Bbb Q)$ be the similitude factor. Then $h_0(N) Z=N^{-1}Z$.
Let $\iota_N=N h_0(N)\iota h_0(N)^{-1}=h_0(N^2)\iota$. Then
$$F_\chi| \iota_N=W(\chi)^2 N^{-1} F_{\bar\chi}.
$$
\end{theorem}

So by taking $Z=\sqrt{-1} Y$, we have
\begin{equation}\label{fun}
F_\chi(\sqrt{-1} (N^2 Y)^{-1})=(-1)^{\frac {3k}2} W(\chi)^2 N^{3k-1} \det(Y)^k F_{\bar\chi}(\sqrt{-1} Y).
\end{equation}

\begin{proof} For $(a,N)=1$, let $\lambda,b\in\Bbb Z$ such that $\lambda N-ab=1$. Then we claim that
\begin{equation}\label{gamma}
p_{\frac aN 1_3} \iota_N=N\iota \gamma p_{\frac bN 1_3}.
\end{equation}
for some $\gamma\in \Gamma$, where $p_B\in \bold N$. We prove it below.
Since 
$$\sum_{a \, \text{(mod $N$)}} \bar\chi(a) \bold e({\rm tr}(T)\frac aN)=\chi({\rm tr}(T)) W(\bar\chi),
$$
we have
$$F_\chi=\frac 1{W(\bar\chi)} \sum_{a \, \text{(mod $N$)}} \bar\chi(a) F|p_{\frac aN 1_3}.
$$
Then 
\begin{eqnarray*}
&& W(\bar\chi) F_\chi| \iota_N= \sum_{a\, \text{(mod $N$)}\atop (a,N)=1} \bar\chi(a) F| p_{\frac aN 1_3} \iota_N = \sum_{a\, \text{(mod $N$)}\atop (a,N)=1} \bar\chi(a) F| p_{\frac bN 1_3} \\
&& \phantom{xxxxxxxxx} =\chi(-1) \sum_{b\, \text{(mod $N$)}\atop (b,N)=1} \chi(b) F| p_{\frac bN 1_3} =\chi(-1) W(\chi) F_{\bar\chi}.
\end{eqnarray*}

Since $W(\chi)W(\bar\chi)=\chi(-1)N$, we have the result.
\end{proof}

\noindent{\it Proof of (\ref{gamma}).}
We show that 
$g=N^{-1}\iota^{-1} p_{\frac aN 1_3} \iota_N p_{\frac bN 1_3}^{-1}\in \Gamma$. Note that $\iota^{-1}=-\iota$ and $p_A^{-1}=p_{-A}$. It is enough to show that $g(X,\xi,X',\xi')\in \bold W_\mathfrak o$ and $g^{-1}(X,\xi,X',\xi')\in \bold W_\mathfrak o$, where 
$\bold W_\mathfrak o=\{(X,\xi,X',\xi')\, | \, X,X'\in \frak J(\Bbb Z), \xi,\xi'\in\Bbb Z\}$.
Recall $\iota(X,\xi,X',\xi')=(-X',-\xi',X,\xi)$ and $h_0(N)^{-1}=h_0(N^{-1})$. Also $p_B'=\iota p_{-B}\iota^{-1}$ is in the opposite unipotent subgroup.
Hence
$g=N p_{-\frac aN 1_3}'p_{-bN 1_3} h_0(N^{-2}$. We can compute $g(X,\xi,X',\xi')=(X_1,\xi_1,X_1',\xi_1')$, where
\begin{eqnarray*}
&& X_1=\frac XN-\frac {b\xi'}{N^2} 1_3-\frac {2a}N 1_3\times (NX'-2b1_3\times X+\frac {b^2\xi'}N 1_3) \\
&& \phantom{xxxxxxxx} +\frac {a^2}{N^2} 1_3(N^3\xi-bN^2(1_3,X')+b^2N(1_3,X)-b^3\xi'),\\
&& \xi_1=N^3\xi-bN^2(1_3,X')+b^2N(1_3,X)-b^3\xi',\\
&& X_1'=NX'-2b1_3\times X+\frac {b^2\xi'}N 1_3-\frac aN 1_3(N^3\xi-bN^2(1_3,X')+b^2N (1_3,X)-b^3\xi'),\\
&& \xi_1'=\frac {\xi}{N^3}-\frac a{N^3} (1_3, X)+\frac {ab\xi'}{N^3} (1_3,1_3)+\frac {a^2}{N^2}(1_3, NX'-2b 1_3\times X+\frac {b^2\xi'}N 1_3) \\
&& \phantom{xxxxxxx} -\frac {a^3}{N^3} (N^3\xi-bN^2(1_3,X')+b^2N (1_3,X)-b^3\xi').
\end{eqnarray*}

We use the fact that $(1_3,X)={\rm tr}(X)$, and $1_3\times X=\frac 12 {\rm tr}(X)1_3-\frac 12 X$. In particular, $(1_3,1_3)=3, 1_3\times 1_3=1_3$.
By using $1+ab=\lambda N$, we obtain our result.
For example, 
$$\xi_1'=-a^3\xi+\frac {\xi'}{N^3}(1+ab)^3-\frac a{N^2} {\rm tr}(X)(1+ab)^2+\frac {a^2}N {\rm tr}(X')(1+ab)\in \Bbb Z. 
$$
Since $g^{-1}=Np_{\frac bN 1_3}\iota_N^{-1}p_{\frac aN 1_3}^{-1} \iota$, we can show in the same way that
$g^{-1}(X,\xi,X',\xi')\in \bold W_\mathfrak o$. $\square$

\begin{remark}
Consider $\phi: \bold G(\Bbb Z)\longrightarrow \bold G(\Bbb Z/N^2\Bbb Z)$.
Let $I$ be the subgroup of $\bold G(\Bbb Z/N^2\Bbb Z)$ generated by $\bold N(\Bbb Z/N^2\Bbb Z)$ and $M_0$, where $M_0$ is the preimage of the scalar matrices under the map $\bold M(\Bbb Z/N^2\Bbb Z)\longrightarrow \bold M(\Bbb Z/N\Bbb Z)$.
Let $\Gamma_0^*(N^2)=\phi^{-1}(I)$. Then $\calm(N;\ZZ)\subset \Gamma_0^*(N^2)$. 

For $\gamma\in \Gamma_0^*(N^2)$, $\gamma\equiv p$ (mod $N$) for some $p\in \bold M(\Bbb Z/N\Bbb Z)$. Define $\chi(\gamma)=\chi(\nu(p))$. It is well-defined. If $G$ is a holomorphic function on $\frak T$ which satisfies
$$G|_k \gamma(Z)=\omega(\gamma) G(Z),\quad Z\in \frak T,\, \gamma\in\Gamma_0^*(N^2),
$$
then $G$ is called a modular form on $\frak T$ of weight $k$ with respect to $\Gamma_0^*(N^2)$ with the central character
$\omega$.

Let $M_k(\Gamma_0^*(N^2),\omega)$ be the space of 
modular forms of weight $k$ with respect to $\Gamma_0^*(N^2)$ on $\frak T$ with the central character $\omega$. 
We can also define the space of cusp forms $S_k(\Gamma_0^*(N^2),\omega)$ using the Siegel $\Phi$-operator. Then we expect 
$F_\chi\in S_k(\Gamma_0^*(N^2),\chi^2)$. We do not need this fact in this paper.
\end{remark}

\section{Twisted Koecher-Maass series of the first kind}
From now on put ${\bf e}(x)=\exp(2\pi \sqrt{-1}x)$ for $x \in \CC$.
Let $F$ be a cusp form of weight $k$ ($k$ even) on the exceptional domain $\frak T$ with respect to $\Gamma$, and let
$F(Z)=\sum_{T \in \frak J(\Bbb Z)_{> 0}} a_F(T){\bf e}((T,Z)).$
Recall the twisted Koecher-Maass series $K^{(1)}(s,F,\chi)$ of the first kind for a Dirichlet character $\chi$;
$$K^{(1)}(s,F,\chi)=\sum_{T \in {\frak J}(\Bbb Z)_{>0}/\calm'(N;\ZZ)} \frac {\chi(\mathrm{Tr}(T)) a_F(T)}{\epsilon_N(T) \det(T)^s}.
$$
In this section we prove the analytic continuation and the functional equation of $K^{(1)}(s,F,\chi)$.

Recall that $d^*Y=det(Y)^{-9}dY$ is the invariant measure in $R_3^+(\Bbb R)$.
Let $\mathcal R$ be a fundamental domain for the action of $\calm(\Bbb Z)$ on $R_3^+(\Bbb R)$.
Let $\mathcal R_N$ be a fundamental domain for the action of $\calm(N;\ZZ)$ on $R_3^+(\Bbb R)$. We may take 
$$\mathcal R_N=\bigcup_{a=1}^r m_a \mathcal R,
$$
where $\{m_1,...,m_r\}$ is a set of representatives for $\calm'(\ZZ)/\calm'(N;\ZZ)$.
We prove

\begin{theorem} \label{KM-general}
For $Re(s)>9+\frac k2$, $K^{(1)}(s,F,\chi)$ converges absolutely, and in this region we have the integral representation
$$\Lambda(s,F,\chi)=\pi^{12}(2\pi)^{-3s} N^{3s} \Gamma(s)\Gamma(s-4)\Gamma(s-8) K^{(1)}(s,F,\chi)=N^{3s} \int_{\mathcal R_N} F_\chi(\sqrt{-1} Y) \det(Y)^{s}\, d^*Y.
$$
It has the analytic continuation to all of $\Bbb C$, and satisfies the functional equation
$$\Lambda(k-s,F,\chi)=(-1)^{\frac {3k}2} \frac {W(\chi)^2}N \Lambda(s,F,\bar\chi).
$$
\end{theorem}

\begin{proof} By Hecke's bound, $a_F(T)\ll det(T)^{\frac k2}$. Hence 
$$K^{(1)}(s,F,\chi)\ll \sum_{T \in {\mathfrak J}(\Bbb Z)_{>0}/\calm(N;\ZZ)} \frac {\det(T)^{\frac k2-Re(s)}}{\epsilon_N(T)}.
$$
By integral test, the series is majorized by 
$$\int_{R_3^+(\Bbb R)\atop \det(Y)\geq 1} \det(Y)^{\frac k2-Re(s)+9}\, d^*Y,
$$
which converges for $Re(s)>9+\frac k2$.

We have $F_\chi(\sqrt{-1} Y)=\sum_{T\in \frak J(\Bbb Z)_{>0}} \chi({\rm tr}(T))a_F(T)e^{-2\pi(T,Y)}$. Since
$a_F(u T)=a_F(T)$ for $u\in \mathcal U_T(\Bbb Z)\cap \calm(N;\ZZ)$, 
\begin{eqnarray*}
&& N^{3s} \int_{\mathcal R_N} F_\chi(\sqrt{-1} Y) \det(Y)^s\, d^*Y= N^{3s} \sum_{T\in \frak J(\Bbb Z)_{>0}} \chi({\rm tr}(T))a_F(T) \int_{\frak R_N} e^{-2\pi (T,Y)} \det(Y)^s\, d^*Y\\
&&\phantom{xxx} =N^{3s}\sum_{T\in \frak J(\Bbb Z)_{>0}/\calm(N;\ZZ)} \frac {\chi({\rm tr}(T))a_F(T)}{\epsilon_N(T)} \sum_{u\in \mathcal U_T(\Bbb Z)\cap \calm(N;\ZZ)} \int_{\mathcal R_N} e^{-2\pi (T,Y)} \det(Y)^s\, d^*Y\\
&&\phantom{xxx} =N^{3s} \sum_{T\in \frak J(\Bbb Z)_{>0}/\calm(N;\ZZ)} \frac {\chi({\rm tr}(T)) a_F(T)}{\epsilon_N(T)} \int_{R_3^+(\Bbb R)} e^{-2\pi (T,Y)} \det(Y)^s\, d^*Y\\
&&\phantom{xxx} =N^{3s} \pi^{12}(2\pi)^{-3s} \Gamma(s)\Gamma(s-4)\Gamma(s-8) \sum_{T\in T\in \frak J(\Bbb Z)_{>0}/\calm(N;\ZZ)} \frac {\chi({\rm tr}(T))a_F(T)}{\epsilon_N(T)} \det(T)^{-s}.
\end{eqnarray*}

Here we use the fact \cite[page 538]{Bai} that for $Re(s)>8$, 

$$\int_{R_3^+(\Bbb R)} e^{-2\pi (T,Y)} \det(Y)^s\, d^*Y=\pi^{12}(2\pi)^{-3s}\Gamma(s)\Gamma(s-4)\Gamma(s-8) \det(T)^{-s}.
$$

For the functional equation, we write 
$$\Lambda(s,F)=N^{3s}\int_{\mathcal R_N\atop \det(Y)\leq N^{-3}} F_\chi(\sqrt{-1} Y) \det(Y)^s\, d^*Y + N^{3s} \int_{\mathcal R_N\atop det(Y)\geq N^{-3}} F_\chi(\sqrt{-1} Y) \det(Y)^s\, d^*Y.
$$
Notice that $\det((N^2Y)^{-1})=N^{-6} \det(Y)^{-1}$. Use the change of variables $Y\longmapsto (N^2Y)^{-1}$, and 
the functional equation (\ref{fun}):
\begin{eqnarray*}
&& N^{3s}\int_{\mathcal R_N\atop \det(Y)\leq N^{-3}} F_\chi(\sqrt{-1} Y) \det(Y)^s\, d^*Y =N^{3s} \int_{\mathcal R_N\atop \det(Y)\geq N^{-3}} F_\chi(\sqrt{-1} N^{-2}Y^{-1}) \det(Y)^{-s}\, d^*Y \\
&&\phantom{xxxxxxxxxxx} = \int_{\mathcal R_N\atop \det(Y)\geq N^{-3}} (-1)^{\frac {3k}2} N^{3k-3s} \frac {W(\chi)^2}N F_{\bar\chi}(\sqrt{-1} Y) \det(Y)^{k-s}\, d^*Y.
\end{eqnarray*}
Hence
$$\Lambda(s,F)=\int_{\mathcal R_N\atop \det(Y)\geq N^{-3}} 
\left( F_\chi(\sqrt{-1} Y)\left(N^3 \det(Y)\right)^s + (-1)^{\frac {3k}2} \frac {W(\chi)^2}N F_{\bar\chi}(\sqrt{-1} Y)\left(N^3 \det(Y)\right)^{k-s} \right) \, d^*Y.
$$
The analytic continuation and the functional equation follow from this.
\end{proof}

\begin{remark} Even though we do not need in this paper, one can ask:
is it true that $\epsilon_N(T)=1$ for $N\geq 3$?
\end{remark}

\section{Relationship between the twisted Koecher-Maass series of the first and the second kind}

In this section we express the twisted Koecher-Maass series of the first kind in terms of the twisted Koecher-Maass series of the second kind. Define $h(A,\chi)$ as
$$h(A,\chi)=\sum_{g \in \calm'(\ZZ)/\calm'(N;\ZZ)} \chi(\mathrm{Tr}(g\cdot A)).$$

\begin{proposition} \label{rewritten-twisted_Koecher-Maass}
 Let 
$$F(Z)=\sum_{T \in \frkJ(\ZZ)_{>0}}a_F(T)\exp(2\pi \sqrt{-1}(T,Z))$$ 
be a cusp form of weight $k$ ($k$ even) with respect to ${\bf G}(\ZZ)$ and $\chi$ a Dirichlet character mod $N$. Then
we have
$$K^{(1)}(s,F,\chi)=\sum_{T \in \frkJ(\ZZ)_{>0}/\calm'(\ZZ)} {h(T,\chi)a_F(T) \over \epsilon(T) (\det T)^{s}}.$$
\end{proposition}
\begin{proof}
This can be proved in the same manner as in \cite[Proposition 3.1]{K-M}. But for readers' convenience, we give a proof. The assertion is trivial if $N=1$. Suppose that $N >1$. We have
\begin{align*}
&K^{(1)}(s,F,\chi)=\sum_{T \in \frkJ(\ZZ)_{>0}/\calm'(\ZZ)} {a_F(T) \over (\det T)^s} \sum_{T' \in \frkJ_A/\calm'(N;\ZZ)} {\chi(\mathrm{Tr}(T')) \over \epsilon_N(T')},
\end{align*}
where 
\[\frkJ_T=\{ T' \in \frkJ(\ZZ)_{>0} \ | \ T' \sim_{\calm'(\ZZ)} T \}.\]
We note that $\epsilon_N(T')=\epsilon_N(T)$ for any $T' \in \frkJ_T$, and 
\[{\epsilon(T) \over \epsilon_N(T)}=\#(\calm'(N;\ZZ)\calu_T(\ZZ)/ \calm'(N;\ZZ)).\]
Fix an element $T$ of $\frkJ(\ZZ)_{>0}$. Let $g_1,g_2 \in \calm'(\ZZ)$. Then $g_1\cdot T \sim_{\calm'(N;\ZZ)} g_2 \cdot T$ if and only if $g_2 \in g_1 \calm'(N;\ZZ) \calu_T(\ZZ) $. Therefore, 
the set $\{g \cdot T \ | \ g \in  \calm'(\ZZ)/ \calm'(N;\ZZ) \calu_T(\ZZ) \}$ is a complete set of representatives of $\frkJ_T/\calm'(N;\ZZ)$. Moreover, we note that $\mathrm{Tr}(g \cdot T) \text{ mod } N$ depends only on $g \calm'(N;\ZZ)\calu_T(\ZZ)$.
Hence we have
\begin{align*}
&\epsilon(T)\sum_{T' \in \frkJ_T/\calm'(N;\ZZ)} {\chi(\mathrm{Tr}(T')) \over \epsilon_N(T')}=\sum_{g \in \calm'(\ZZ)/\calm'(N;\ZZ)\calu_T(\ZZ)  } \chi(\mathrm{Tr}(g \cdot T)){\epsilon(T) \over \epsilon_N(T)}\\
&\phantom{xxxxxxxx} =\sum_{g \in \ \calm'(\ZZ)/ \calm'(N;\ZZ)\calu_T(\ZZ)} \chi(\mathrm{Tr}(g \cdot T))\#(\calm'(N;\ZZ)\calu_T(\ZZ)/ \calm'(N;\ZZ))\\
&\phantom{xxxxxxxx} =\sum_{g \in \calm'(\ZZ)/ \calm'(N;\ZZ)} \chi(\mathrm{Tr}(g \cdot T)).
\end{align*}
This proves the assertion. 
\end{proof}
\begin{remark} \label{rem.comment-on-K-M-Prop.3.1}
\begin{itemize}
\item[(1)] We need not assume that $\epsilon_N(T)=1$ for $N>1$ unlike \cite[Proposition 3.1]{K-M}.
\item[(2)] There is a typo in the proof of \cite[Proposition 3.1]{K-M}. The equality `$\mathrm{tr}(A[U_1])=\mathrm{tr}(A[U_2])$' on page 463. line 7 should be `$\mathrm{tr}(A[U_1]) \equiv \mathrm{tr}(A[U_2]) \text{ mod } N$.'
\end{itemize}
\end{remark}

Let $\chi$ be a Dirichlet character mod $N.$ Fix a prime factor $p$ of $N.$ For an integer $n$ prime to $p,$ take an integer $m$ such that 
$$m \equiv \begin{cases} n \text{ mod }  p^e \\
1 \text{ mod } N/p^e
\end{cases}.
$$
We then put  
$$\chi^{(p)}(n)= 
\begin{cases} \chi(m) & \text{ if }  (n,p)=1\\
0 & \text{ if } (n,p)\not=1
\end{cases}.$$
Then it is independent of the choice of $m,$ and $\chi^{(p)}$ is a character mod $p^e,$ and we have $\chi=\prod_{p|N} \chi^{(p)}.$ 

For Dirichlet characters $\chi_1,\ldots,\chi_r$ mod $N$, we define the generalized Jacobi sum $J(\chi_1,\ldots,\chi_r)$  by 
$$J(\chi_1,\ldots,\chi_r)=\sum_{a_1,\ldots,a_r \in \ZZ/N\ZZ \atop a_1+\cdots+a_r=1} \chi_1(a_1)\cdots \chi_r(a_r).
$$
We note that $J(\chi_1,\chi_2)$ is the usual Jacobi sum.
Moreover, we define the Jacobi sum 
$J_{\frkJ_2}(\chi_1,\chi_2)$ on $\frkJ_2$ by 
$$J_{\frkJ_2}(\chi_1,\chi_2)=\sum_{B \in \frkJ_2(\ZZ/N\ZZ)} \chi_1(\det B)\chi_2(1-\mathrm{Tr}(B)).$$
\begin{proposition} \label{prop.decomposition-character-sum}
\begin{itemize}
\item[(1)] Let $A \in \frkJ(\ZZ)$ and $\chi$ a Dirichlet character mod $N$. Then we have
$$h(A,\chi)=\prod_{p | N} h(A,\chi^{(p)}).$$
\item[(2)] Let $\chi_1, \ldots, \chi_r$ be Dirichlet characters mod $N$.
\begin{itemize}
\item[(2.1)] Suppose that $r \ge 3$ and that $\chi_1 \cdots  \chi_{r-1}$ is primitive. Then,
$$J(\chi_1,\ldots,\chi_r)=J(\chi_1 \cdots \chi_{r-1},\chi_r)J(\chi_1,\ldots,\chi_{r-1}),$$
\item[(2.2)] Suppose that $\chi_1 \chi_2$ is primitive. Then,
$$J(\chi_1,\chi_2)={W(\chi_1)W(\chi_2) \over W(\chi_1\chi_2)}.$$
\end{itemize}
\end{itemize}
\end{proposition}
\begin{proof} The assertion (1) can be proved by the Chinese remainder theorem. 
To prove (2.1) and (2.2), again by the Chinese remainder theorem, we may assume that $N=p^m$ with $p$ a prime number.
We have
\begin{align*}
&J(\chi_1,\ldots,\chi_r)=\sum_{a_1,\ldots,a_{r-2},a_r \in \ZZ/p^m\ZZ} \chi_1(a_1)\cdots\chi_{r-2}(a_{r-2})\chi_{r-1}(1-a_1-\cdots-a_{r-2}-a_r)\chi_r(a_r) \\
&=\sum_{a_1,\ldots,a_{r-2} \in \ZZ/p^m \ZZ, v \in \ZZ/p^{m-1}\ZZ} \chi_1(a_1)\cdots \chi_{r-2}(a_{r-2})\chi_{r-1}(-pv-a_1-\cdots-a_{r-2})\chi_r(pv-1)\\
&\phantom{xxxxxx} +\sum_{a_1,\ldots,a_{r-2},a_r \in \ZZ/p^m\ZZ \atop a_r  \not\equiv 1 \text{ mod } p} \chi_1(a_1)\cdots \chi_{r-2}(a_{r-2})\chi_{r-1}(1-a_1-\cdots-a_{r-2}-a_r)\chi_r(a_r).
\end{align*}
Since $\chi_1\cdots \chi_{r-1}$ is primitive, the first term of the right-hand side of the above equation is zero. Hence, putting  $a_1=(1-a_r)b_1,\ldots,a_{r-2}=(1-a_r)b_{r-2}$ in the second term of the right-hand side of the above equation, we have 
\begin{align*}
J(\chi_1,\cdots,\chi_r)&=\sum_{b_1,\ldots,b_{r-2},a_r \in \ZZ/p^m\ZZ} 
 \chi_1((1-a_r)b_1)\cdots \chi_{r-2}((1-a_r) b_{r-2})\\
&\phantom{xxx}  \times  \chi_{r-1}(1-a_r-(1-a_r)(b_1+\cdots+b_{r-2}))\chi_r(a_r)\\
&=\sum_{b_1,\ldots,b_{r-2} \in \ZZ/p^m\ZZ } \chi_1(b_1)\cdots \chi_{r-2}(b_{r-2})\chi_{r-1}(1-b_1-\cdots -b_{r-2}) \\
&\phantom{xxx} \times \sum_{a_r \in \ZZ/p^m\ZZ} (\chi_1 \cdots \chi_{r-1})(1-a_r)\chi_r(a_r).
\end{align*}
This proves the assertion (2.1). The second assertion (2.2) is well known in the case $m=1$ (e.g.  \cite[Ch.8, Theorem 1]{IR}), and the general case can also be proved in the same manner. 
\end{proof}
Recall from the introduction, $\cald_N=\{ \eta \ | \ \eta \text{ is a Dirichlet character } mod \ N \text{ such that } \eta^l=1 \}$, where $l={\rm GCD}(3,\phi(N))$.
\begin{corollary} \label{cor.explicit-generalized-Jacobi-sum}
Let $\widetilde \chi$ be a Dirichlet character mod $N$ and put $\chi=\widetilde \chi^3$. Suppose that $N$ is odd and $\chi$ is primitive. Then, for any $\eta \in \cald_{N}$, we have
\begin{align*}
J(\overline{\widetilde \chi \eta},\overline{\widetilde \chi \eta},\overline{\widetilde \chi \eta})={W(\overline{\widetilde \chi \eta})^3 \over W(\bar \chi)}.
\end{align*}
\end{corollary}
\begin{proof} By the assumption, $\chi^2$ is also primitive, and so is $\overline{\widetilde \chi^2 \eta^2}$ for any $\eta \in \cald_{N}$. Hence, by Lemma \ref{lem.explicit-generalized-Jacobi-sum}, and Proposition \ref{prop.decomposition-character-sum}, we have
\begin{align*}
J_{\frkJ_2}(\overline{\widetilde \chi \eta},\overline{\widetilde \chi \eta})&=N^4J(\widetilde \chi \eta, \widetilde \chi \eta,\widetilde \chi \eta) =N^{4}J(\overline{\widetilde \chi \eta},(\overline{\widetilde \chi\eta})^2)J(\overline{\widetilde \chi \eta},\overline{\widetilde \chi \eta})\\
&=N^{4}{W(\overline{\widetilde \chi \eta}) W((\overline{\widetilde \chi\eta})^2)\over W((\overline{\widetilde \chi \eta})^3)}{W(\overline{\widetilde \chi \eta})W(\overline{\widetilde \chi \eta}) \over W((\overline{\widetilde \chi \eta})^2)}=N^4{W(\overline{\widetilde \chi \eta})^3 \over W(\bar \chi)}.
\end{align*}
\end{proof}
\begin{remark} \label{rem.non-primitive}
For a primitive character $\widetilde \chi$ mod $2^m$, $\widetilde \chi^2$ is not primitive. 
\end{remark}

We will show in Lemma \ref{lem.explicit-generalized-Jacobi-sum} that $J_{\frkJ_2}(\chi_1,\chi_2)=N^4 J(\chi_1,\chi_1,\chi_2)$.

For a commutative ring $R$, let $S_l(R)$ denote the set of symmetric matrices of size $l$ with entries in $R$. 
For $S \in S_l(R)$ and $X \in M_{lm}(R)$, put
$S[X]={}^t X SX$. Then, $GL_l(R)$ acts on $S_l(R)$ in the following way:
\[GL_l(R) \times S_l(R)  \ni (g,S) \longmapsto S[{}^tg].\]
Let $A_1$ and $A_2$ be elements of $S_l(R)$. Then, by definition,  we have  $A_1 \sim_{GL_l(R)} A_2$ if there exists an element $g \in GL_l(R)$ such that $ A_1[{}^t g]=A_2$. (This is equivalent to saying that there exists an element $g \in GL_l(R)$ such that $ A_1[g]=A_2$.)
Let $R$ be an integral domain and $K$ its quotient field.  Suppose that the characteristic of $K$ is different from $2$. We denote by $\calh_l(R)$ the set of half-integral symmetric matrices of size $l$ over $R$, that is, the set of symmetric matrices $A=(a_{ij})$ in $S_l(K)$ such that $2a_{ij}, a_{ii} \in R$. 

Let $S \in \calh_l(\ZZ_p)$ with $l$ even and $\det S \not=0$. Then we put 
$$\chi(S)=\chi_p(S)=
\begin{cases} 1 & \text{ if $S \sim_{GL_l(\ZZ_p)} \overbrace{H \bot \cdots \bot H}^{l/2}$} \\
-1 & \text{ otherwise}
\end{cases},
$$
where $H=\begin{pmatrix} 0 & 1/2 \\ 1/2 & 0 \end{pmatrix}$. 
For $w\in \frkC_{\QQ}$, let $N(w) \in \QQ$ be the norm. Then $N$ defines a quadratic form over $\QQ$ in $8$ variables, and in particular, 
$N|\frak o$ defines an integral quadratic form. Let $\cals_N$ be the Gram matrix of $N|\frak o$
with respect to the standard basis. By construction,  $\cals_N \in \calh_8(\ZZ) \cap {1 \over 2} GL_8(\ZZ)$ and positive definite. Hence for any prime number $p$, we have
$$\cals_N \sim_{GL_8(\ZZ_p)} H \bot H \bot H \bot H.$$
Hence $\chi_p(\cals_N)=1$. 
For  $Z_1 =\begin{pmatrix} a &  b \\ \bar b & d \end{pmatrix} \in \frkJ_2(\QQ)$ and $w=\begin{pmatrix} w_1 \\ w_2 \end{pmatrix} \in M_{2,1}(\frkC_\QQ)$, $Z_1[w]:={}^t \bar wZ_1w$ is well-defined as the usual matrix multiplication.
Then 
$$Z_1[w]=aN(w_1)+dN(w_2)+\mathrm{Tr}(\bar w_1 b w_2),$$
and 
$$Q: M_{2,1}(\frkC_\QQ) \longrightarrow \Bbb Q,\quad w \mapsto Z_1[w],
$$
defines a quadratic form over $\QQ$ in $16$ variables. In particular if $Z_1 \in \frkJ_2(\ZZ)$, then $Q|_{M_{2,1}(\frak o)}$ defines an integral quadratic form.

For $A \in \calh_l(\ZZ)$ and $c \in \ZZ$, let
$$\cala_{p^m}(A,c)=\{x \in M_{l1}(\ZZ/p^m\ZZ) \ | \ A[x] \equiv c \text{ mod } p^m \},
$$
and 
$$\cala_{p^m}^{\rm prim}(A,c)=\{x \in \cala_{p^m}(A,c) \ | \ x \not\equiv 0 \text{ mod } p \},$$
and put $A_{p^m}(A,c)=\#\cala_{p^m}(A,c)$ and $A_{p^m}^{\rm prim}(A,c)=\#\cala_{p^m}^{\rm prim}(A,c)$.

\begin{lemma} \label{lem.quadratic-equation}
Let $l$ be a positive even integer. 
Let $S \in \calh_l(\ZZ) \cap {1 \over 2}GL_l(\ZZ_p)$ and $c \in \ZZ$. 
\begin{itemize}
\item[(1)] We have 
$$A_p^{\rm prim}(S,c)=\begin{cases} p^{l-1}(1-p^{-l/2}\chi(S)), & \text{ if } c \not\equiv 0 \text{ mod } p,\\
p^{(l-1)}(1-p^{-l/2}\chi(S)) (1+p^{-l/2+1}\chi(S)), & \text{ if } c \in p\ZZ.
\end{cases}$$
\item[(2)] We have 
$$A_p(S,c)=\begin{cases} p^{l-1}(1-p^{-l/2}\chi(S)), & \text{ if } c \not\equiv 0 \text{ mod } p, \\
p^{(l-1)}(1-p^{-l/2}\chi(S)) (1+p^{-l/2+1}\chi(S)) +1, & \text{ if } c \in p\ZZ.
\end{cases}
$$
\item[(3)] Suppose that $m \ge 2$. Then
$$A_{p^m}(S,c)=A^{\rm prim}(S,c)+p^lA_{p^{m-2}}(S,p^{-2}c).$$
Here we put the convention that $A_{p^{m-2}}(S,p^{-2}c)=0$ if $\ord_p(c)\le 1$, and $A_{p^{m-2}}(S,p^{-2}c)=1$ if $m=2$ and $\ord_p(c)\ge 2$.
\end{itemize}
\end{lemma}
\begin{proof}
The assertion (2) follows from \cite[Theorem 1.3.2]{Ki2} and \cite[Lemma 5.5.9]{Ki2}. The assertion (1) follows from \cite[Theorem 1.3.2]{Ki2} and \cite[Lemma 1.3.1]{Ki2}.
Suppose that $m \ge 2$. Then
\[\cala_{p^m}(S,c)=A_{p^m}^{\rm prim}(S,c) \cup \cala_{p^m}'(S,c),\]
where 
\[\cala_{p^m}'(S,c) =\{x \in \cala_{p^m}(S,c) \ | \ x \equiv 0 \text{ mod } p \}.\]
Put 
\[\cala_{p^m}''(S,c) =\{x \in M_{l1}(\ZZ_p/p^{m-1}\ZZ_p) \ | \ A[x] \equiv p^{-2}c  \text{ mod } p^{m-2}.\]
Then clearly we have
\[\#\cala_{p^{m-1}}''(S,p^{-2}c) =p^l \# \cala_{p^{m-2}}(A,p^{-2}c).\]
For $x = py \in \cala_{p^m}'(A,c)$, the element $y \text{ mod } p^{m-1}$  belongs to
$\cala_{p^{m-1}}''(S,p^{-2}c)$, and the mapping $x \mapsto y \text{ mod } p^{m-1}$ gives a bijection from $\cala_{p^m}'(A,c)$ to $\cala_{p^{m-1}}''(S,p^{-2}c)$.
This proves the assertion (3).
\end{proof}

\begin{corollary} \label{cor.quadratic-equation} 
Fix $S \in \calh_l(\ZZ) \cap {1 \over 2}GL_l(\ZZ_p)$. Then, $A_{p^m}(S,c)$ depends only on $\ord_p(c)$.
\end{corollary}

\begin{lemma} \label{lem.quadratic-sum}
Let $l,S$ and $c$ be as in Lemma \ref{lem.quadratic-equation}.
Let $\eta$ be a primitive character mod $p^m.$ Put
$$I_{\eta,S,c}=\sum_{{\bf w} \in (\ZZ/p^m\ZZ)^l} \eta(S[^{t}{\bf w}]+c).$$
  Then
$$I_{\eta,S,c}=p^{lm/2} \eta(c) \times \begin{cases}
\chi(S) & \text{ if } m \text{ is odd}, \\
1 & \text{ if } m \text{ is even.}
\end{cases}
$$
\end{lemma}
\begin{proof} Suppose that $m \ge 2$. Then, by Lemma \ref{lem.quadratic-equation}, we have 
\begin{align*}
I_{\eta,S,c}&=\sum_{u \in \ZZ/p^m\ZZ} \eta(u) A_{p^m}(S,u-c)\\
&=p^{l-1}(1-p^{-l/2}\chi(S))\sum_{u \in \ZZ/p^m\ZZ\atop u-c \not\equiv 0 \text{ mod } p} \eta(u)\\
&+\sum_{u \in \ZZ/p^m\ZZ\atop u-c \equiv 0 \text{ mod } p}\Big(p^{l-1}(1-p^{-l/2}\chi(S))(1+p^{-l/2+1}\chi(S))+p^lA_{p^{m-2}}(S,p^{-2}(u-c)\Big)\eta(u).
\end{align*}
Since $\eta$ is primitive, we have
\begin{align*}
I_{\eta,S,c}=p^l \sum_{u \in \ZZ/p^m\ZZ\atop u-c \equiv 0 \text{ mod } p^2} A_{p^{m-2}}(S,p^{-2}(u-c))\eta(u)=p^l \sum_{v \in \ZZ/p^{m-2}\ZZ} A_{p^{m-2}}(S,v) \eta(c+p^2v).
\end{align*}
Suppose that $c \in p\ZZ$. Then $\eta(c+p^2v)=0$ for any $v$, and therefore $I_{\eta,S,c}=0$.
Suppose that $c \not\in p\ZZ$. 
For $r \le m/2$, put
$$L(r)=\sum_{v \in \ZZ/p^{m-2r}\ZZ} A_{p^{m-2r}}(S,v) \eta(c+p^{2r}v).$$
Then, $I_{\eta,S,c}=p^l L(1)$. Using the same argument as above and by Corollary \ref{cor.quadratic-equation}, we have 
$$L(r)=p^lL(r+1)$$
if $2r+2 \le m$. Hence we have
$$I_{\eta,S,c}=\begin{cases} p^{ml/2} L(m/2) & \text{ if } m \text{ is } even, \\
p^{(m-1)l/2} L((m-1)/2) & \text{ if } m \text{ is odd.}
\end{cases}
$$
By definition, we have $L(m/2)=\eta(c)$. The relation $L((m-1)/2)=p^{l/2}\chi(S))\eta(c)$ follows from  \cite[Lemma 5.3]{Kat2} in the case $p$ is odd, and it can also be proved in the case $p=2$ in the same way.  
This proves the assertion.
\end{proof}

\begin{lemma} \label{lem.explicit-generalized-Jacobi-sum}
Let $\chi$ and $\eta$ be  Dirichlet characters mod $N$. Suppose that $\chi$ is primitive. Then
$$J_{\frkJ_2}(\chi,\eta)=N^4J(\chi,\chi,\eta).$$ 
\end{lemma}
\begin{proof}
By Chinese remainder theorem, we may assume that $N=p^m$ with $p$ a prime number.
By definition,
\begin{align*}
&J_{\frkJ_2}(\chi,\eta)=\sum_{(z_{11},z_{22}) \in (\ZZ/p^m\ZZ)^2} \sum_{
z_{12} \in \frko/p^m \frko} \chi(-N(z_{12})+z_{11}z_{22})\eta(1-z_{11}-z_{22}).
\end{align*}
Here, the norm $N(z_{12})$ can be regarded as a quadratic form over $\ZZ/p^m\ZZ$ of rank $8$ with determinant $1$. Hence, by Lemma \ref{lem.quadratic-sum}, we have
\begin{align*}
&J_{\frkJ_2}(\chi,\eta)=\sum_{(z_{11},z_{22}) \in (\ZZ/p^m\ZZ)^2} p^{4m}\chi(z_{11}z_{22})\eta(1-z_{11}-z_{22}).
\end{align*}
This proves the assertion. 
\end{proof}

For a positive integer $N$, let
$d_N=N^{64}\prod_{p|N} (1-p^{-2})(1-p^{-6})(1-p^{-8})(1-p^{-12})$ as in Theorem \ref{th.explicit-KM-1}.
For $T \in \frkJ(\ZZ/N\ZZ)$ and a positive integer $N$, let 
$$\phi_{T,N}: \calm'(\ZZ/N\ZZ) \ni g \mapsto g \cdot T \in \frkJ(\ZZ/N\ZZ).$$
\begin{lemma} \label{lem.number-of-fiber}
Let $p$ be a prime number and  $m$ a positive integer.
Let $T, S \in \frkJ(\ZZ/p^m\ZZ)$ such that $\det T \in (\ZZ/p^m\ZZ)^\times$.
 Then $\phi_{T,p^m}^{-1}(S) \not=\emptyset$ if and only if $\det S=\det T$. 
Moreover
$$\#(\phi_{T,p^m}^{-1}(S))=p^{52m}(1-p^{-2})(1-p^{-6})(1-p^{-8})(1-p^{-12}).$$
\end{lemma}
\begin{proof} The proof will be given after Proposition \ref{prop.reinterpolation-of-local-density}.
\end{proof}
The following lemma is easy to prove.
\begin{lemma} \label{lem.existence-of-third-root-of-character}
Let $l=\mathrm{GCD}(3,\phi(N))$ and let $u_0$ be a primitive $l$-th root of unity mod $N$. Let $\chi$ be a Dirichlet character mod $N$ and suppose that $\chi(u_0)=1$. Then there exists a Dirichlet character $\widetilde \chi$ such that $\chi=\widetilde  \chi^3$.
\end{lemma}

\begin{theorem} \label{th.explicit-character-sum}
Let $\chi$ be a primitive character mod $p^m$. Suppose that $\chi$ is not a quadratic character. 
Let $l=\mathrm{GCD}(3,\phi(p^m))$ and let $u_0$ be a primitive $l$-th root of unity mod $p^m$.
\begin{itemize}
\item [(1)] Suppose that $\chi(u_0) \not=1$. Then $h(A,\chi)=0$.
\item [(2)] Suppose that $\chi(u_0)=1$. Fix a character $\widetilde \chi$ such that $\widetilde \chi^3=\chi$. Then
$$h(A,\chi)=d_{p^m} \sum_{\eta \in \cald_{p^m}} (\widetilde \chi \eta)(\det A)J(\overline{\widetilde \chi \eta},\overline{\widetilde \chi \eta}, \overline{\widetilde \chi \eta}).
$$
In particular if $p$ is odd, then
$$h(A,\chi)=d_{p^m} \sum_{\eta \in \cald_{p^m}} (\widetilde \chi \eta)(\det A){W(\overline{\widetilde \chi \eta})^3 
\over W(\bar \chi)}.
$$
\end{itemize}
\end{theorem}
\begin{proof}
The case $\chi=1$ is trivial. Suppose that $\chi \not=1$.

First, suppose that $\det A \in p\ZZ$. We may assume that $A=A_0 \bot p A_1$ with $A_0 \in \frkJ_2(\ZZ/p^m\ZZ)$ and $A_1 \in \frkC(\ZZ/p^m\ZZ)$.
First let $p$ be odd. Then there exists an element $\xi_0 \in \ZZ/p^m\ZZ$ such that $\xi_0 \equiv 1 \text{ mod } p^{m-1}$ and $\chi(\xi_0)^2\not=1$.  Indeed,
if $m=1$, this follows just from the assumption. If $m \ge 2$, since $\chi$ is primitive mod $p^m$, there exists an element $\xi_0$ such that $\xi_0 \equiv 1 \text{ mod } p^{m-1}$ and $\chi(\xi_0) \not=1$.  Then we have $\xi_0^p \equiv 1 \text{ mod } p^m$, and hence we can easily see that $\chi(\xi_0)^2 \not=1$.
We have  $\theta(1,1,\xi_0^{-3})\cdot A=A$ and $\nu(\theta(1,1,\xi_0^3)=\xi_0^{-6}$. Define $g_0 \in \calm(\ZZ/p^m\ZZ)$ as
$$g_0:\frkJ(\ZZ/p^m\ZZ) \ni  (t_{ij}) \mapsto (\xi_0^2t_{ij}) \in \frkJ(\ZZ/p^m\ZZ).$$
Then, $\nu(g_0) =\xi_0^6$, and hence $g_0g\theta(1,1,\xi_0^{-3})=1$. Hence
\begin{align*}
h(A,\chi)&=\sum_{g \in \calm'(\ZZ/p^m\ZZ)} \chi(\mathrm{Tr}((g_0g\theta(1,1,\xi_0^{-3}))\cdot A)\\
&=\sum_{g \in \calm'(\ZZ/p^m\ZZ)} \chi(\mathrm{Tr}(g_0\cdot(g\cdot A))=\chi(\xi_0^2)h(A,\chi).
\end{align*}
Hence we have $h(A,\chi)=0$. Next, let $p=2$. Then, $m \ge 4$. Put $\xi_0=1+2^{m-1} \text{ mod } 2^m$ and $\zeta_0=1+2^{m-2} \text{ mod } 2^m$. Then,
$\zeta_0^2=\xi_0$ and $\zeta_0^6=\xi_0^3=\xi_0$. We have 
$\theta(1,1,\zeta_0^{-1})\cdot A=A$ and $\nu(\theta(1,1,\zeta_0^{-1}))=\xi_0^{-1}$. Then, similarly as above we have
\begin{align*}
h(A,\chi)&=\sum_{g \in \calm'(\ZZ/p^m\ZZ)} \chi(\mathrm{Tr}((g_0g\theta(1,1,\zeta_0^{-1}))\cdot A)\\
&=\sum_{g \in \calm'(\ZZ/p^m\ZZ)} \chi(\mathrm{Tr}(g_0\cdot(g\cdot A))=\chi(\xi_0)h(A,\chi).
\end{align*}
Hence we have $h(A,\chi)=0$.

Next, suppose that $\det A$ is a $p$-unit. For $c \in \ZZ/p^m\ZZ,$ put
 $$\calr_{p^m}(A,c)=\{g \in \calm'(\ZZ/p^m\ZZ) \ | \ \mathrm{Tr}(g\cdot A)=c\}. $$ 
Then we have 
$$h(A,\chi)=\sum_{c \in \ZZ/p^m\ZZ}\chi(c) \#(\calr_{p^m}(A,c)).$$ 
Let
$$\cals_{p^m}(A,c)=\{B \in \frkJ(\ZZ/p^m\ZZ) \ | \ \det B =\det A \text{ and } \mathrm{Tr}(B)=c \}.$$
Then, by Lemma \ref{lem.number-of-fiber}, 
$$\#(\calr_{p^m}(A,c))=a_{p^m} \#(\cals_{p^m}(A,c)),$$
where $a_{p^m}=p^{52m}(1-p^{-2})(1-p^{-6})(1-p^{-8})(1-p^{-12})$. Let 
$$\widetilde \cals_{p^m}(A,c)=\{ (Z_1,w) \in  \frkJ_2(\ZZ/p^m\ZZ) \times \in M_{2,1}(\frkC(\ZZ/p^m\ZZ)) 
\ | \ \det \begin{pmatrix} Z_1 & w \\ {}^t \bar w & 1-\mathrm{Tr}(Z_1) \end{pmatrix} c^3 =\det A \}.
$$
Then we have 
$$\#(\widetilde \cals_{p^m}(A,c))=\#(\cals_{p^m}(A,c)).$$
First, suppose that $\chi(u_0) \not=1$. Then we have
$$\widetilde \cals(A,cu_0)=\widetilde \cals(A,c),
$$
for any $c \in (\ZZ/p^m\ZZ)^\times$, and hence 
$$\sum_{c \in \ZZ/p^m\ZZ} \chi(c)\#\widetilde \cals(A,c)=\sum_{c \in \ZZ/p^m\ZZ} \chi(u_0c)\#\widetilde \cals(A,c)=\chi(u_0)\sum_{c \in \ZZ/p^m\ZZ} \chi(c)\#\widetilde \cals(A,c).$$
Therefore, $h(A,\chi)=0.$

Suppose that $\chi(u_0)=1$. Then we have
$$h(A,\chi)=a_{p^m} \sum_{(Z_1,w_1) } \overline{\widetilde \chi}\left(\det  \begin{pmatrix} Z_1 & w \\ {}^t \bar w &1-{\rm tr}(Z_1) \end{pmatrix} \right),
$$
where $(Z,w)$ runs over all elements of $\frkJ_2(\ZZ/p^m\ZZ) \times M_{2,1}(\frkC(\ZZ/p^m\ZZ))$
such that 
\begin{equation}\label{det}
\det  \begin{pmatrix} Z_1 & w \\ {}^t \bar w &1-{\rm tr}(Z_1) \end{pmatrix} \equiv u^3 \text{ mod } p^m,
\end{equation}
with some $u \in \ZZ \cap \ZZ_p^\times$. For such a matrix 
$\begin{pmatrix} Z_1 & w \\ {}^t \bar w &1-{\rm tr}(Z_1) \end{pmatrix}$, there exist exactly $l$ elements satisfying $(\ref{det})$. We have
$$\sum_{\eta \in \cald_{p^m}} \widetilde \chi \eta(v)=l\widetilde \chi(v) \text{ or } 0,
$$
according as $v \equiv u^m \text{ mod } p^m $ with some $u \in \ZZ \cap \ZZ_p^\times$ or not. Hence we have
$$h(A,\chi)=a_{p^m} \sum_{\eta \in \cald_{p^m}} \, 
\sum_{(Z_1,w) \in  \frkJ_2(\ZZ/p^m\ZZ) \times M_{2,1}
(\frkC(\ZZ/p^m\ZZ))) } \overline  {(\widetilde \chi \eta)}(\det A) \overline{(\widetilde \chi \eta)}\left(\det \begin{pmatrix} Z_1 & w \\ {}^t\bar w & 1-\mathrm{Tr}(Z_1)\end{pmatrix}\right).
$$
For $Z_1$, put
$$I(Z_1)=\sum_{w \in  M_{2,1}(\frkC(\ZZ/p^m\ZZ)) } \overline  {(\widetilde \chi \eta)}(\det A) \overline{(\widetilde \chi \eta)}\left(\det \begin{pmatrix} Z_1 & w \\ {}^t\bar w & 1-\mathrm{Tr}(Z_1)\end{pmatrix}\right).
$$
Then
$$I(Z_1)=\sum_{w \in  M_{2,1}(\frkC(\ZZ/p^m\ZZ)) } \overline{(\widetilde \chi \eta)}\Bigl(-\mathrm{Ad}(Z_1)[w] + (\det  Z_1) (1-\mathrm{Tr}(Z_1))\Bigr),
$$
where $\mathrm{Ad} (Z_1)=\begin{pmatrix} z_{22} & -z_{12} \\ -\overline{z_{12}} & z_{11} \end{pmatrix}$ for $Z_1=\begin{pmatrix} z_{11} & z_{12} \\ \overline{z_{12}} & z_{22} \end{pmatrix}$.
Then we may suppose that $Z_1=O$ or $Z_1=\xi_1 \bot p\xi_2$ with $\xi_1, \xi_2 \in \ZZ$ such that $\xi_1 \not\in p\ZZ$. In the former case, clearly we have $I(Z_1)=0$. In the latter case, $I(Z_1)$ can be expressed as
$$I(Z_1)=\sum_{w_2}\sum_{w_1} \overline{(\widetilde \chi \eta)}\Bigl(-N(w_1)+ (\det Z_1)(1-\mathrm {Tr}(Z_1))-p\xi_1 N(w_2)\Bigr).
$$
Then the Gram matrix of the quadratic form $-N(w)$ is $-\cals_N$ and $\chi(-\cals_N)=1$.
We note that $(\det Z_1)(1-{\rm tr}(Z_1))-p\xi_1 N(w_2) \in p\ZZ_p$. Thus by Lemma \ref{lem.quadratic-sum}, we have $I(Z_1)=0$. Suppose that $\det Z_1 \in \ZZ_p^\times$. Then we may suppose that $Z_1=\xi_1 \bot \xi_2$ with $\xi_1,\xi_2 \in \ZZ_p^\times$ and
$$I(Z_1)=\sum_{w \in  M_{2,1}(\frkC(\ZZ/p^m\ZZ))} \overline{(\widetilde \chi \eta)}\Bigl(-\xi_1N(w_1)-\xi_2N(w_2) + (\det  Z_1)(1-\mathrm{Tr}(Z_1))\Bigr).
$$
Then the Gram matrix of the quadratic form $-\xi_1N(w_1)-\xi_2N(w_2)$ is $-\xi_1\cals_N \bot \xi_2\cals_N$, and we have
$\chi(-\xi_1\cals_N \bot \xi_2\cals_N)=1$. Thus, by Lemma \ref{lem.quadratic-sum},
$$I(Z_1)=p^{8m}\overline{(\widetilde \chi \eta)}((\det Z_1)(1-\mathrm{Tr}(Z_1)).
$$
Hence we have
$$h(A,\chi)=a_{p^m}p^{8m} \sum_{\eta \in \cald_{p^m}} (\widetilde \chi \eta)(\det A)J_{\frkJ_2}(\overline{\widetilde \chi \eta},\overline{\widetilde \chi \eta})).
$$
Thus the assertion follows from Lemma \ref{lem.explicit-generalized-Jacobi-sum} and Corollary \ref{cor.explicit-generalized-Jacobi-sum}.
\end{proof}

\begin{corollary} \label{cor.explicit-character-sum}
Let $\chi$ be a primitive character mod $N$. Suppose that $\chi$ is not a quadratic character. 
Let $l=\mathrm{GCD}(3,\phi(N))$ and let $u_0$ be a primitive $l$-th root of unity mod $N$.
\begin{itemize}
\item [(1)] Suppose that $\chi(u_0) \not=1$. Then $h(A,\chi)=0$.
\item [(2)] Suppose that $\chi(u_0)=1$. Fix a character $\widetilde \chi$ such that $\widetilde \chi^3=\chi$. Then
$$h(A,\chi)=d_{N} \sum_{\eta \in \cald_{N}} (\widetilde \chi \eta)(\det A)J(\overline{\widetilde \chi \eta},\overline{\widetilde \chi \eta}, \overline{\widetilde \chi \eta} ).
$$
In particular, if $N$ is odd, then 
$$h(A,\chi)=d_{N} \sum_{\eta \in \cald_{N}} (\widetilde \chi \eta)(\det A){W(\overline{\widetilde \chi \eta})^3 \over W(\bar \chi)}.
$$
\end{itemize}
\end{corollary}
\begin{proof} Let $N=p_1^{e_1}\cdots p_r^{e_r}$ with $p_1,\ldots,p_r$ distinct primes such that $e_1,\ldots,e_r$ are positive integers. Then the mapping
$$\cald_N \ni \chi \mapsto (\chi^{(p_1)},\ldots,\chi^{(p_r)}) \in \cald_{p_1^{e_1}} \times \cdots \times \cald_{p_r^{e_r}},
$$
is a bijection. Hence the assertion follows from  Proposition \ref{prop.decomposition-character-sum} and 
Theorem \ref{th.explicit-character-sum}.
\end{proof}

\begin{theorem} \label{explicit-twisted-Koecher-Maass-first-kind}
Let $\chi$ be a primitive  character mod $N$. Suppose that $\chi$ is not a quadratic character. 
Let $l=\mathrm{GCD}(3,\phi(N))$ and let $u_0$ be a primitive $l$-th root of unity mod $N$.
\begin{itemize}
\item [(1)] Suppose that $\chi(u_0) \not=1$. Then 
$$K^{(1)}(s,F,\chi)=0.
$$
\item [(2)] Suppose that $\chi(u_0)=1$. Fix a character $\widetilde \chi$ such that $\widetilde \chi^3=\chi$. Then
$$K^{(1)}(s,F,\chi)=d_{N} \sum_{\eta \in \cald_{N}} K^{(2)}(s,F,\widetilde \chi \eta)J(\overline{\widetilde \chi \eta},\overline{\widetilde \chi \eta}, \overline{\widetilde \chi \eta}).
$$
In particular if $N$ is odd, then
$$K^{(1)}(s,F,\chi)=d_{N} \sum_{\eta \in \cald_{N}} K^{(2)}(s,F,\widetilde \chi \eta){W(\overline{\widetilde \chi \eta})^3 \over W(\bar \chi)}.
$$
\end{itemize}
\end{theorem}

\section{Mass formula for the exceptional group of type $F_4$}\label{F_4}

In this section, we recall the mass formula and local density from \cite{K-K-Y} for the sake of completeness.
For $T \in \frak J_R$, we define a group scheme  $\calu_T$ over $R$ by
\[\calu_T(S)=\{ g \in  \calm(S) \ | \ g\cdot T=T \}\]
for any commutative $R$-algebra $S$. 
By definition, $\calu_T(S)\subset \calm'$. 
In particular, for $T \in \frkJ(\ZZ)_{>0}$, put ${\bf U}_T=\calu_T \otimes_{\ZZ} \QQ$. It is easy to see that 
${\bf U}_T$ is a connected regular algebraic group over $\QQ$ by (geometric) fiberwise argument. 
Further, ${\bf U}_T$ is
an exceptional group of type $F_4$ \cite[p.108]{Mars}, and therefore we call $\calu_T$ the group scheme of type $F_4$.

Let $G$ be a group (resp. a group scheme) acting on a set (resp. a scheme) $S$. Then, for $s \in S$, we denote by $\calo_G(s)$ the orbit (resp. the orbit scheme) of $s$ under $G$, that is
\[\calo_G(s)=\{g\cdot s \ | \ g \in G \}.\]
For $T \in \calm(\ZZ_p)^{\rm ns}$ we note that we have  $\int_{\calo_{\calm(\ZZ_p)}(T)} |d\sigma(x)|_p \not=0$.

Recall the definition of the local density  $\beta_p(T)$ of $T$ for $T \in \calm(\ZZ_p)^{\rm ns}$ \cite{K-K-Y}: 
\[\beta_p(T)={ (1-p^{-1})\delta_p \over  \int_{\calo_{\calm(\ZZ_p)}(T)} |d\sigma(x)|_p },\]
where
$\delta_p= (1-p^{-2})(1-p^{-5})(1-p^{-6})(1-p^{-8})(1-p^{-9})(1-p^{-12})$. 

\begin{proposition} \label{prop.reinterpolation-of-local-density}
For $T \in \frkJ(\ZZ_p)^{\rm ns}$ and an integer $n \ge \ord_p(\det T)+1$, we have 
\[\beta_p(T)= p^{-52n}\#
\calu_T(\ZZ_p/p^n\ZZ_p).\]
\end{proposition}

\begin{proof}
The assertion follows from \cite[Lemmas 3.3 and 6.8, Theorem 3.7]{K-K-Y}.
\end{proof}

\noindent{\it Proof of Lemma \ref{lem.number-of-fiber}.} The first assertion follows from \cite[Lemma 3.1 (2.2)]{K-K-Y}.
Suppose that $\phi_{T,p^m}^{-1}(S)) \neq \emptyset$. Then, clearly we have
\[\#\phi_{T,p^m}^{-1}(S))=\#\phi_{T,p^m}^{-1}(T))=\# \calu_T(\ZZ_p/p^m\ZZ_p),\]
and the second assertion follows from \cite[Corollary 6.2, (1)]{K-K-Y}. $\square$

\smallskip

Let $T$ be an element of $\frkJ(\ZZ)_{>0}$. For $T' \sim_{\calm'(\ZZ)} T$, we say that $T'$ belongs to the same ${\bf M}_\AAA$-genus as $T$ and write
 $T' \approx T$ if $T' \sim_{\calm'(\ZZ_p)} T$ for any prime number $p$.
For $T \in {\mathfrak J}(\ZZ)_{>0}$, let 
\[\calg(T)=\{T'\in {\mathfrak J}(\ZZ)_{>0} \ | \ T' \approx T \}.\]  Put 
\begin{equation*}\label{mass-def}
\mathrm{Mass}(T)=\sum_{T' \in \calg(T)/\calm'(\ZZ)} {1 \over \epsilon(T')},
\end{equation*}
where $\epsilon(T')=\#\calu_{T'}(\ZZ)$.

Then we have the mass formula for $T$ (cf. \cite[Theorem 3.8]{K-K-Y}).

\begin{theorem} \label{th.mass-formula2}(Mass-formula)
Let $T$ be an element of $\frkJ(\ZZ)_{>0}$. Then we have
\[\mathrm{Mass}(T)= c {(\det T)^9 \over \prod_{p<\infty} \beta_p(T)},\quad c=\frac {5! 7! 11!}{(2\pi)^{28}}.\]
\end{theorem}

For $p \le \infty$, let $\iota_p:\frkJ(\QQ) \longrightarrow \frkJ(\QQ_p)$ be the natural embedding, and 
let $\varphi: \frkJ(\QQ) \longrightarrow   \prod_{p \le \infty} \frkJ(\QQ_p)$ be the diagonal embedding.
Let
\[\JJ=\prod_p (\frkJ(\ZZ_p)/\calm'(\ZZ_p)).\]
Then $\varphi$  induces a mapping from $\frkJ(\ZZ)_{>0}/\prod_p \calm'(\ZZ_p)$ to $\JJ$, which will be denoted also by $\varphi$. For  $d \in \ZZ_p \setminus \{0\}$, put
\[\frkJ(d,\ZZ_p)=\{ T \in  \frkJ(\ZZ_p) \ | \ \det T=d\}.\]
Moreover, for a positive integer $d$, put
\[\frkJ(d, \ZZ)=\{ T \in \frkJ(\ZZ) \ | \ \det T=d\},\]
and
\[\JJ(d)=\prod_p (\frkJ(d,\ZZ_p)/\calm'(\ZZ_p)).\]
Now we have the following local-global principle (cf. \cite[Proposition 3.10]{K-K-Y}).
\begin{proposition} \label{prop.local-global-exceptional-Jordan} The mapping $\varphi$ induces a bijection from
$\frkJ(d,\ZZ)_{>0}/\prod_p \calm'(\ZZ_p)$ to $\JJ(d).$
\end{proposition}

\section{Twisted Koecher-Maass series of the second kind of the Ikeda type lift for the exceptional group of type $E_{7,3}$}\label{KM}

We review the Ikeda type lift of a cuspidal Hecke eigenform in \cite{K-Y}, and 
consider its twisted Koecher-Maass series of the second kind. Let $k\geq 10$ be a positive integer, and let 
\[f(\tau)=\sum_{m=1}^{\infty} a_f(m)\exp(2\pi \sqrt{-1} m\tau)\]
be in $S_{2k-8}(SL_2(\ZZ))$. For a prime number $p$, let  
$\alpha_p$ be a complex number  such that  and 
$a_f(p)=p^{(2k-9)/2}(\alpha_p+\alpha_p^{-1}).$ By Deligne's theorem we have $|\alpha_p|=1$. For a Dirichlet character $\chi$ we define the automorphic $L$-function $L(s,\pi_f,\chi)$ of the cuspidal representation $\pi_f$ attached to $f$ as 
\[L(s,\pi_f,\chi)=\prod_p \{(1-p^{-s}\alpha_p\chi(p))(1-p^{-s}\alpha_p^{-1}\chi(p))\}^{-1}.\]
If $\chi$ is the principal character, we simply write  $L(s,\pi_f,\chi)$ as $L(s,\pi_f)$.

Let $p$ be a prime number. For $T \in \frkJ(\QQ_p)$, let $T \sim_{\calm(\ZZ_p/p^n\ZZ_p)}  \epsilon_1 p^{a_1} \bot \epsilon_2 p^{a_2} \bot \epsilon_3 p^{a_3}$ with $a_1,a_2,a_3\in\Bbb Z\cup \{\infty\}$, $a_1\leq a_2\leq a_3$, and $\epsilon_i\in \Bbb Z_p^\times$. 
Define $\kappa_p(T)$ by $\displaystyle \kappa_p(T)=\prod_{1 \le i \le 3 \atop a_i>0} p^{a_i}$. Here we make the convention that $\kappa_p(T)=1$ if $T=O$. We note that $\kappa_p(T)$ is uniquely determined by 
$T \text{ mod } \frkJ(\ZZ_p)$. Moreover, for $x \in \QQ_p$, put $\mathbf{e}_p(x)=\exp(2\pi \sqrt{-1}\mathrm{Frac}(x))$, where $\mathrm{Frac}(x)$ is the fractional part of $x$.
For $T \in \frkJ(\ZZ_p)^{\rm ns}$,  let $S_p(T)$ be the local Siegel series defined by
\[S_p(s,T)=\sum_{T' \in \frkJ(\QQ_p)/\frkJ(\ZZ_p)} {\bf e}_p((T,T'))\kappa_p(T')^{-s}.\]
Then, there is a polynomial $f_T^p(X)$ in $X$ such that 
\[S_p(s,T)=(1-p^{-s})(1-p^{4-s})(1-p^{8-s})f_T^p(p^{9-s}).\]
Put
\[\widetilde f_T^p(X)=X^{\ord_p(\det T)}f_T^p(X^{-2}).\]
Then it satisfies the functional equation
\begin{equation}\label{functional}
\widetilde f_T^p(X^{-1})=\widetilde f_T^p(X).
\end{equation}
For $T \in \frkJ(\ZZ)_{>0}$ put
$a_{F_f}(T)=\det(T)^{\frac {2k-9}2} \prod_{p| \det(T)} \widetilde f_T^p(\alpha_p)$, and define the Fourier series $F_f(Z)$ on $\frkT$ by
$$F_f(Z)=\sum_{T\in \frkJ(\ZZ)_{>0}} a_{F_f}(T) {\bf e}((T,Z)) \quad  (Z \in \frkT).
$$
Then, Kim and Yamauchi \cite{K-Y} showed that
$F_f$ is a cuspidal Hecke eigenform of weight $2k$ for ${\bf G}(\ZZ)$.

We consider the twisted Koecher-Maass series of $F_f$ of the second kind.
For a Dirichlet character $\chi$ mod $N$, recall 
$$K^{(2)}(s,F_f,\chi)=\sum_{T\in {\frak J}(\Bbb Z)_{>0}/\calm'(\ZZ)} \frac {\chi(\det T)a_{F_f}(T)}{\epsilon(T) \det(T)^s}.
$$
Even though $K^{(2)}(s,F,\chi)$ does not have an Euler product for a general $F$, we show that $K^{(2)}(s,F_f,\chi)$ has an Euler product, which enables us to reduce its computation to each $p$-adic place. For $d \in \ZZ_p \setminus \{0\}$, put
\[\lambda_p(d,X)=\sum_{T \in \frkJ(d,\ZZ_p)/\calm'(\ZZ_p)} {\widetilde f_T^p(X) \over  \beta_p(T)},\]
and for a positive integer $d$, put
\[C(d;f)=\prod_{p<\infty} \lambda_p(d,\alpha_p).\]

\begin{theorem} \label{th.local-global-KM1}
We have
\[K^{(2)}(s,F_f,\chi)=c \sum_{d=1}^{\infty}C(d,f)\chi(d)d^{-s+k+\frac 92},\]
where $c$ is as in Theorem \ref{th.mass-formula2}.
\end{theorem}
\begin{proof}
Let $\calg =\frkJ(\ZZ)_{>0}/\!\approx$ be the set of all genera of $\frkJ(\ZZ)_{>0}$.
We note the Fourier coefficient $a_{F_f}(T)$ is uniquely determined by $\calg(T)$. Hence, by the Mass formula \cite[Theorem 3.8]{K-K-Y}, we have
\begin{align*}
K^{(2)}(s,F_f,\chi)&=\sum_{T \in \calg}\sum_{T' \in \calg(T)/\calm'(\ZZ)} (\det T')^{-s} {a_{F_f}(T') \chi(\det T')\over \epsilon(T')}=\sum_{T \in \calg}\mathrm{Mass}(T) \frac {a_{F_f}(T) \chi(\det T)}{(\det T)^{s}} \\
&=c\sum_{T \in \calg}  (\det T)^{k+\frac 92-s}\chi(\det T)\prod_p { \widetilde f_T^p(\alpha_p) \over \beta_p(T)}\\
&=c\sum_{d=1}^{\infty}  d^{-s+\frac 92+k} \chi(d)\sum_{ T \in \frkJ(d,\ZZ)/\prod\calm'(\ZZ_p)} \prod_p{ \widetilde f_T^p(\alpha_p) \over \beta_p(T)}.
\end{align*}
Thus the assertion follows from \cite[Proposition 3.10]{K-K-Y}.
\end{proof}

For $d \in \ZZ_p^{\times}$, define a formal power series $H_p(d;X,t)$ as
\[H_p(d;X,t)=\sum_{m=0}^{\infty} \lambda_p(p^md,X)t^{m}.\]
As in \cite[Lemma 5.2]{K-K-Y}, we can show that
$\lambda_p(d;X)$ is determined by $\ord_p(d)$, and hence $H_p(d;X,t)$ does not depend on the choice of $d \in \ZZ_p^{\times}$, and
we write it as $H_p(X,t)$. 
Therefore, 
\begin{theorem} \label{th.local-global-KM2}
We have
\[K^{(2)}(s,F_f,\chi)=c\prod_p H_p(\alpha_p,\chi(p)p^{-s+9/2+k}).\]
\end{theorem}

Recall the formula for $\widetilde f_T^p(X)$ from \cite[Corollary 7.2]{K-K-Y}:
For $T=p^{m_1} \bot p^{m_1+m_2} \bot p^{m_1+m_3}$ with $0 \le m_1, 0  \le m_2 \le m_3$, 
\begin{align*}
&\widetilde f_T^p(X)={X^{-m_2-m_3-3m_1} \over (1-X^2)(1-p^{4}X^2)(1-p^8X^2)}+ {X^{m_2+m_3+3m_1} \over (1-X^{-2})(1-p^{4}X^{-2})(1-p^8X^{-2})} \\
&\phantom{xxxx} -{p^{8m_1+8}X^{-m_1-m_2-m_3+2} \over (1-X^2)(1-p^{4}X^2)(1-p^8X^2)}- {p^{8m_1+8}X^{m_1+m_2+m_3-2} \over (1-X^{-2})(1-p^{4}X^{-2})(1-p^8X^{-2})} \\
&\phantom{xxxx} -{p^{8m_1+4(m_2+1)}X^{-m_3+m_2-m_1+2} \over (1-X^2)^2(1-p^4X^2)}-{p^{8m_1+4(m_2+1)}X^{m_3-m_2+m_1-2} \over (1-X^{-2})^2(1-p^4X^{-2})}\\
&\phantom{xxxx} -{p^{8m_1+4m_2}X^{m_3-m_2-m_1+2} \over (1-X^2)^2(1-p^{-4}X^2)}-{p^{8m_1+4m_2}X^{-m_3+m_2+m_1-2} \over (1-X^{-2})^2(1-p^{-4}X^{-2})}.
\end{align*}

Now we have
\begin{theorem} \label{th.explicit-local-KM}
\begin{align*}
H_p(X,t) =\{(1-p^{-2})(1-p^{-6})(1-p^{-8})(1-p^{-12})\}^{-1} {1 \over \prod_{i=1}^3 (1-p^{-4i+3}X^{-1}t)(1-p^{-4i+3}X t)}.
\end{align*}
\end{theorem}
\begin{proof}
For $i=1,2,3,4$, put
\begin{align*}
&A_1(X)=\{(1-X^2)(1-p^{4}X^2)(1-p^8X^2)\}^{-1} \\
&A_2(X)=-p^8X^2\{(1-X^2)(1-p^{4}X^2)(1-p^8X^2)\}^{-1} \\
&A_3(X)=-p^4X^2\{(1-X^2)^2(1-p^4X^2)^{-1} \\
&A_4(X)=-X^2\{(1-X^2)^2(1-p^{-4}X^2)\}^{-1} ,
\end{align*}
and for $i=5,6,7,8$, put $A_i(X)=A_{i-4}(X^{-1})$.
For $i=1,2,3,4$, we also define $X_i=X_i(X),Y_i=Y_i(X),Z_i=Z_i(X)$ as
\begin{align*}
& X_1=X^{-3}, \quad X_2=X_3=X_4=p^8X^{-1}, \\
& Y_1=Y_2=X^{-1},\quad Y_3=p^4X,\quad Y_4=p^4X^{-1}, \\
& Z_1=Z_2=Z_3=X^{-1}, \quad Z_4=X,
\end{align*}
and for $i=5,6,7,8$, put $X_i(X)=X_{i-4}(X^{-1})$, $Y_i(X)=Y_{i-4}(X^{-1})$, and $Z_i(X)=Z_{i-4}(X^{-1})$. Then, as in \cite[Theorem 7.4]{K-K-Y}, we have
\begin{align*}
&H_p(X,t)=\sum_{1 \le i \le 8} A_i(X)P(X_i,Y_i,Z_i,t),
\end{align*}
where
\[P(X_i,Y_i,Z_i,t)=\sum_{m_1 \ge 0,0 \le m_2 \le m_3} {t^{3m_1+m_2+m_3} X_i^{m_1}Y_i^{m_2}Z_i^{m_3} \over \beta_p(p^{m_1} \bot p^{m_1+m_2} \bot p^{m_1+m_3})}\]
for $i=1,\ldots,8$.
By \cite[Lemma 7.3]{K-K-Y}, we have
\begin{align*}
&P(X_i,Y_i,Z_i,t)=\{(1-p^{-2})(1-p^{-6})(1-p^{-8})(1-p^{-12})\}^{-1}\\
&\phantom{xxxxxxxxxxx} \times {1+(p^{-5}+p^{-9})t Z_i+(p^{-14}+p^{-8})t^2Y_iZ_i+p^{-23}t^3Y_iZ_i^2 \over 
(1-p^{-27}X_it^3)(1-p^{-10}Y_iZ_it^2)(1-p^{-1}Z_i t)}.
\end{align*}
From now on, put $P_i(X,t)=(1-p^{-2})(1-p^{-6})(1-p^{-8})(1-p^{-12})P_i(X_i,Y_i,Z_i,t)$. Then,
\begin{align*}P_1(X,t)={1 \over (1-p^{-9}X^{-1}t)(1-p^{-5}X^{-1}t)(1-p^{-1}X^{-1}T)}, \tag{6.1}
\end{align*}
\[P_2(X,t)={1+p^{-9}X^{-1}t+p^{-18}X^{-2}t^2 \over (1-p^{-19}X^{-1}t^3)(1-p^{-5}X^{-1}t)(1-p^{-1}X^{-1}t)}.\]
Moreover, we have
\[P_3(X,t)={1 +(p^{-5}+p^{-9})t X^{-1}+(p^{-10}+p^{-14})t^2  +p^{-19}t^3 X^{-1} \over (1-p^{-19}X^{-1}t^3)(1-p^{-6}t^2)(1-p^{-1}X^{-1}t)},\]
\[P_4(X,t)={1 +(p^{-5}+p^{-9})t X+(p^{-10}+p^{-14})t^2  +p^{-19}t^3 X \over (1-p^{-19}X^{-1}t^3)(1-p^{-6}t^2)(1-p^{-1} Xt)}.\]
We prove the following equality.
\begin{align*}
&A_3(X)P_3(X,t)+A_4(X)P_4(X)=-X^2\{(1-X^2)(1-p^{-4}X^2)(1-p^4X^2)\}^{-1}\\
&\phantom{xxxxxxxxx} \times {(1+p^4)(1+p^{-14}t^2)+(X^{-1}+X)p^{-5}t \over (1-p^{-19}X^{-1}t^3)(1-p^{-1}X^{-1}t)(1-p^{-1}Xt)}.
\end{align*}
To prove this, put
\begin{align*}
&Q_{2,3}(X,t)=-X^{-2}(1-X^2)^2(1-p^4 X^2)(1-p^{-4}X^2)\\
&\phantom{xxxxx} \times (1-p^{-19}X^{-1}t^3)(1-p^{-6}t^2)(1-p^{-1}X t)(1-p^{-1}X^{-1}t)\\
&\phantom{xxxxx} \times  (A_3(X)P_3(X,t)+A_4(X)P_4(X,t)).
\end{align*}
Then, by a simple computation, we have
\begin{align*}
&Q_{2,3}(X,t)=(p^4-X^2)(1-p^{-1}X t)\Big( 1 +(p^{-5}+p^{-9})t X^{-1}+(p^{-10}+p^{-14})t^2  +p^{-19}t^3 X^{-1} \Big)\\
&\phantom{xxx} +(1-p^4 X^2)(1-p^{-1}X^{-1}t) \Big(1 +(p^{-5}+p^{-9})t X+(p^{-10}+p^{-14})t^2  +p^{-19}t^3 X\Big)\\
&\phantom{xxx} =(p^4-X^2) \Big( 1+((p^{-5}+p^{-9}) X^{-1}-p^{-1}X)t+(p^{-14}-p^{-6})t^2 \\
&\phantom{xxx} +(p^{-19}X^{-1}-(p^{-11}+p^{-15})X)t^3-p^{-20}t^4 \Big)\\
&\phantom{xxx} +(1-p^4X^2) \Big( 1+((p^{-5}+p^{-9}) X-p^{-1}X^{-1})t+(p^{-14}-p^{-6})t^2 \\
& \phantom{xxx} +(p^{-19}X-(p^{-11}+p^{-15})X^{-1})t^3-p^{-20}t^4 \Big)\\
&\phantom{xxx} =(1-X^2)\Big\{(1+p^4)\Big(1+(p^{-14}-p^{-6})t^2-p^{-20}t^4\Big) 
+(X^{-1}+X)\Big(p^{-5}t-p^{-11}t^3 \Big)\Big\}\\
&\phantom{xxx} =(1-X^2)(1-p^{-6}t^2) \Big((1+p^4)(1+p^{-14}t^2)+(X^{-1}+X)p^{-5}t) \Big).
\end{align*}
This proves the above equality.
Hence, again by a simple computation,  we have
\begin{align*}
&A_2(X)P_2(X,t)+A_3(X)P_3(X,t)+A_4(X)P_4(X) \tag{6.2} \\
&=-X^2\{(1-X^2)(1-p^{-4}X^2)(1-p^8X^2)\}^{-1} {(1 +p^4+p^8)  \over (1-p^{-5}X^{-1}t)(1-p^{-1}X^{-1}t)(1-p^{-1}Xt)}. 
\end{align*}
Similarly we have
\begin{align*}
P_5(X,t)={1 \over (1-p^{-9}Xt)(1-p^{-5}X t)(1-p^{-1} X t)}, \tag{6.3}
\end{align*}
\begin{align*}
&A_6(X)P_6(X,t)+A_7(X)P_7(X,t)+A_8(X)P_8(X) \tag{6.4} \\
&=-X^{-2}\{(1-X^{-2})(1-p^{-4}X^{-2})(1-p^8X^{-2})\}^{-1} {1 +p^4+p^8   \over (1-p^{-5}X t)(1-p^{-1}X^{-1}t)(1-p^{-1}Xt)}. 
\end{align*}
Put
\begin{align*}
K_p(X,t)= (1-p^{-2})(1-p^{-6})(1-p^{-8})(1-p^{-12}) H_p(X,t) \prod_{i=1}^3 (1-p^{-4i+3}X^{-1}t)(1-p^{-4i+3}X t).
\end{align*}
 Then, by (6.1), (6.2), (6.3) and (6.4), we have
\begin{align*}
&K_p(X,t)={(1-p^{-9} Xt)(1-p^{-5}X t)(1-p^{-1} Xt) \over (1-X^2)(1-p^4 X^2)(1-p^8 X^2)} \\
&\phantom{xxxxxx} -{X^2(1+p^4+p^8)(1-p^{-9} Xt)(1-p^{-9}X^{-1} t)(1-p^{-5} Xt) \over (1-X^2)(1-p^{-4} X^2)(1-p^8 X^2)}\\
&\phantom{xxxxxx} +{(1-p^{-9} X^{-1}t)(1-p^{-5}X^{-1} t)(1-p^{-1} X^{-1}t) \over (1-X^{-2})(1-p^4 X^{-2})(1-p^8 X^{-2})} \\
&\phantom{xxxxxx} -{X^{-2}(1+p^4+p^8)(1-p^{-9} Xt)(1-p^{-9}X^{-1} t)(1-p^{-5} X^{-1}t) \over (1-X^{-2})(1-p^{-4} X^{-2})(1-p^8 X^{-2})}.
\end{align*}
We note that $K_p(X,t)$ is a polynomial in $t$ of degree at most $3$, and $K_p(X,p^9X^{\pm 1})=K_p(X,p^5X^{\pm 1})=1$.
Therefore, $K_p(X,t)=1$ as a polynomial  in $t$. This proves the assertion.
\end{proof}

\section{Proof of Theorems \ref{th.explicit-KM-2} and \ref{th.explicit-KM-1} and some rationality result}

Theorem \ref{th.explicit-KM-2} and \ref{th.explicit-KM-1} are immediate consequences of Theorems 4.10, \ref{th.local-global-KM2}, and \ref{th.explicit-local-KM}. 

By the functional equation of $L(s,\pi_f,\chi)$, we have the functional equation of $K^{(1)}(s,F_f,\chi)$ of the form $s \longrightarrow  2k-s$, which is compatible with that of the general case (Theorem \ref{KM-general}).

From the mass formula (Theorem \ref{th.mass-formula2}), 
$c={5!7!11! \over (2\pi)^{28}}$. So
$c\zeta(2)\zeta(6)\zeta(8)\zeta(12)=\frac {691}{2^{15}\cdot 3^6\cdot 5^2\cdot 7^2\cdot 13}\in\Bbb Q$. 

Let $L(s,f,\chi)$ be the unnormalized $L$-function, and $K_f, K_\chi$ be Hecke fields.
Then since $L(s,\pi_f,\chi)=L(s+k-\frac 92,f,\chi)$, from Theorem 1.1,
$$
K^{(2)}(s,F_f,\chi) =c \zeta(2)\zeta(6)\zeta(8)\zeta(12)\times \prod_{i=1}^3 L(s+4i-12,f,\chi).
$$
Recall the following rationality result of Shimura \cite{Sh}:
\begin{theorem} \label{th.rationality-Hecke-L}
For a Dirichlet character $\chi$, let $A(m,f,\chi)=(2\pi \sqrt{-1})^{-m} W(\chi)^{-1} L(m,f,\chi)$, and $u^+=A(2k-9,f,\phi), u^-=A(2k-9,f,\phi')$, where $\phi,\phi'$ are any fixed real odd (even, resp.) characters. Then
$$
A(m,f,\chi)\in \begin{cases} u^+ K_f K_\chi, &\text{if $\chi(-1)=(-1)^m$}\\ u^- K_f K_\chi, &\text{if $\chi(-1)=(-1)^{m-1}$}\end{cases},
$$
for every positive integer $m<2k-8$, and
$\pi W(\chi)\langle f,f\rangle \sqrt{-1}\in u^+ u^- K_f$.
\end{theorem}

Therefore, we have

\begin{theorem} \label{th.rationality-KM} 
Let $\chi$ be a Dirichlet character. 
For every integer $m$, $9\leq m\leq 2k-9$, 
$$(2\pi\sqrt{-1})^{-3m+12}W(\chi)^{-3} K^{(2)}(m,F_f,\chi)\in \begin{cases} (u^+)^3 K_f K_\chi, &\text{if $\chi(-1)=(-1)^m$}\\ (u^-)^3 K_f K_\chi, &\text{if $\chi(-1)=(-1)^{m-1}$}\end{cases}.
$$
\end{theorem}

For $K^{(1)}(s,F_f,\chi)$, we assume that $\chi$ is a primitive character mod $N$, and $\chi=\widetilde\chi^3$ for some $\widetilde\chi\in \widehat{(\ZZ/N\ZZ)^{\times}}$. Notice that if $\eta^3=1$, $\eta(-1)=1$. Hence

\begin{theorem} Let $\chi$ be a primitive character mod $N$. Then, for every integer $m$, $9\leq m\leq 2k-9$, 
$$(2\pi\sqrt{-1})^{-3m+12}W(\bar \chi) K^{(1)}(m,F_f,\chi)\in \begin{cases} (u^+)^3 K_f K_{\widetilde\chi}\Bbb Q(\sqrt{-3}), &\text{if $\widetilde\chi(-1)=(-1)^m$}\\ (u^-)^3 K_f K_{\widetilde\chi}\Bbb Q(\sqrt{-3}), &\text{if $\widetilde\chi(-1)=(-1)^{m-1}$}\end{cases}.
$$
In particular if $\phi(N)$ is not divisible by $3$, then 
$$(2\pi\sqrt{-1})^{-3m+12}W(\bar \chi) K^{(1)}(m,F_f,\chi)\in\begin{cases} (u^+)^3 K_f K_{\widetilde\chi}, &\text{if $\widetilde\chi(-1)=(-1)^m$}\\ (u^-)^3 K_f K_{\widetilde\chi}, &\text{if $\widetilde\chi(-1)=(-1)^{m-1}$}\end{cases}.$$
\end{theorem}

\end{document}